\documentclass[12pt,reqno]{amsart}
\usepackage[final]{graphicx}
\usepackage{amsfonts}
\usepackage{pdfsync}
\usepackage{amsmath}
\usepackage{amssymb}
\usepackage{amsthm}
\usepackage{cite}
\usepackage{hyperref}
\hypersetup{
    colorlinks,%
    citecolor=red,%
    filecolor=purple,%
    linkcolor=blue,%
    urlcolor=blue
}

\newtheorem{theorem}{Theorem}[section]
\newtheorem{lemma}[theorem]{Lemma}
\newtheorem{corollary}[theorem]{Corollary}

\newtheorem{remark}[theorem]{Remark}

\newcommand{\R}{{\mathord{\mathbb R}}}
\newcommand{\C}{{\mathord{\mathbb C}}}

\newcommand{\Sp}{{\mathord{\mathbb S}}}

\newcommand{\re}{{\mathop{\rm Re }}}

\begin{document}
\title{Which magnetic fields support a zero mode?}
\author[Rupert L. Frank and Michael Loss]{}

\email{{rlfrank@caltech.edu}}
\email{{loss@math.gatech.edu}}
\thanks{\copyright~2020 by the authors. Reproduction of this article by any means permitted for non-commercial purposes}

\maketitle

\centerline{\scshape Rupert L. Frank and Michael Loss}
\smallskip

{\footnotesize
 \centerline{Department of Mathematics, California Institute of Technology}
 \centerline{Pasadena, CA 91125, United States of America}

 \centerline{\& Department of Mathematics, LMU Munich}
\centerline{Theresienstr. 39, 80333 M\"unchen, Germany}

\centerline{\& Munich Center for Quantum Science and Technology}
\centerline{Schel\-ling\-str.~4, 80799 M\"unchen, Germany}
}

{\footnotesize
 \centerline{School of Mathematics, Georgia Institute of Technology}
 \centerline{Atlanta, GA 30332, United States of America}
}
\begin{abstract} 
	This paper presents some results concerning the size of magnetic fields that support zero modes for the three dimensional Dirac equation and related problems for spinor equations.  It is a well known fact that for the Schr\"odinger in three dimensions to have a negative energy bound state, the  $3/2$ norm of the potential has to be greater than the Sobolev constant. We prove an analogous result for the existence of zero modes, namely that the $3/2$ norm of the magnetic field has to greater than twice the Sobolev constant. The novel point here is that the spinorial nature of the wave function is crucial. It leads to an improved diamagnetic inequality from which the bound is derived. While the results are probably not sharp, other equations are analyzed where the
	results are indeed optimal.
\end{abstract}


\section{Introduction}

Zero modes for the three dimensional Dirac equation play a role in several areas of physics and mathematics,
mostly in an obstructionist way. In physics they show up by creating non-perturbative effects for fermionic determinants in QED \cite{Fry}. They also cause difficulties for the semiclassical energy asymptotics of atoms interacting with magnetic fields (see \cite{ErSo1} and \cite{erSo2}).  They can render matter unstable; fortunately in a region for the physical parameters that is far away from the ones occurring in nature.  Since this is the context in which zero modes were discovered we describe the situation in a bit more detail.

Electrons carry spin, but in the absence of magnetic fields this does not enter the non-relativistic Schr\"odinger equation which determines the dynamics of atoms in any way. Spin, together with the exclusion principle determines the symmetry type of the spatial part of the wave function.
In the presence of magnetic fields, however, it is necessary to include the spin-magnetic field interaction. As an example, the quadratic form of the Hamiltonian describing a hyrogenic atom  (in suitable units) is given by
$$
(\psi, H(A) \psi)= \Vert \sigma \cdot (-i\nabla - A)\psi \Vert_2^2 - Z \int_{\R^3} \frac{|\psi|^2}{|x|} dx =
 \Vert (-i\nabla - A)\psi \Vert_2^2 -(\psi, \sigma \cdot B \psi)  - Z \int_{\R^3} \frac{|\psi|^2}{|x|} dx \ ,
$$
where $\psi(x) = \left(\begin{array}{c} \psi_1(x) \\ \psi_2(x) \end{array} \right)$ is a normalized $2$-spinor, 
$$
\Vert\psi \Vert_2^2= \int_{\R^3} |\psi(x)|^2 dx = \int_{\R^3} [|\psi_1(x)|^2 + |\psi_2(x)|^2] dx = 1\ ,
$$
$A$ is the vector potential with $B={\rm curl }A$, $\sigma$ denotes the vector of Pauli matrices
$$
\sigma_1 = \left( \begin{array}{cc} 0 & 1 \\1 & 0 \end{array}\right) \ , \ \sigma_2=  \left( \begin{array}{cc} 0 & -i \\ i & 0 \end{array}\right) \ ,  \ \sigma_3=\left( \begin{array}{cc} 1 & 0 \\ 0 & -1 \end{array}\right)
$$
and $Z$ is the nuclear charge number.
It is easy to see that the ground state energy of this system, through a suitable choice of the magnetic field can be made arbitrarily negative. It is physically reasonable to try to stabilize the problem by adding the field energy, i.e., one considers
$$
\mathcal E(\psi, A) = (\psi, H(A) \psi)+ \frac{1}{8\pi \alpha^2} \int_{\R^3} |B(x)|^2 dx \,.
$$
This functional should then be  minimized with respect to $A$ and the normalized  $\psi$. If the infimum of this functional is $-\infty$ then we say that the problem is unstable. Here $\alpha \approx \frac{1}{137}$ is the Sommerfeld fine-structure constant. That $\alpha$  appears in such an odd place is due to the choice of units.
This variational problem is analyzed in \cite{FLL} where it is shown that instability implies the existence of zero modes.
The zero mode equation is
\begin{equation} \label{zeromode}
\sigma\cdot(-i\nabla-A(x)) \psi = 0 \ ,
\end{equation}
where the spinor $\psi$ is normalized and  the magnetic field $B={\rm curl } A$ has finite energy. To see how zero modes come into play consider the rescaled zero mode pair
$$
\psi_\lambda(x) = \lambda^{3/2} \psi(\lambda x) \ , \ A_\lambda(x) = \lambda A(\lambda x)  \ .
$$
Obviously  $\sigma \cdot (-i\nabla - A_\lambda) \psi_\lambda = 0$ and $\Vert \psi_\lambda \Vert_2= \Vert \psi \Vert_2$.
The magnetic field $B = {\rm curl }A$ scales as $B_\lambda(x) = \lambda^2 B(\lambda x)$ and a simple calculation yields
$$
\mathcal E (\psi_\lambda, A_\lambda) = \lambda\left( - Z \int_{\R^3} \frac{|\psi|^2}{|x|} dx +  \frac{1}{8\pi \alpha^2} \int_{\R^3} |B(x)|^2 dx\right) \ ,
$$
which shows that for $Z\alpha^2$ large enough, the energy can be made arbitrarily negative by letting $\lambda$ tend to infinity. 
In \cite{FLL} it was also shown that the critical $Z$ beyond which there is collapse is given by
\begin{equation} \label{criticalzee}
Z_c = \inf \left\{ \frac{1}{8\pi \alpha^2} \frac{\int_{\R^3} |B(x)|^2 dx}{ \int_{\R^3} \frac{|\psi|^2}{|x|} dx} \right\}
\end{equation}
where the infimum is taken over all normalized spinors $\psi$ and all magnetic fields with finite energy
that satisfy the zero mode equation \eqref{zeromode}.

The first example of a zero mode pair was given in \cite{LossYau} and it is the following:
\begin{equation} \label{lossyau1}
\psi(x) = \frac{1+i \sigma \cdot x}{(1+|x|^2)^{3/2}} \phi_0
\end{equation}
and
\begin{equation} \label{lossyau2}
A(x) = \frac{3}{(1+|x|^2)^2} \left( (1-|x|^2) w +2 (x \cdot w) x + 2 x \wedge w\right),
\end{equation}
where $\phi_0$ is a constant spinor and $w = \langle \phi_0, \sigma \phi_0\rangle$. Here $\langle\,,\, \rangle$ is the inner product in $\C^2$. The magnetic field is given by
\begin{equation} \label{lossyau3}
B(x) = {\rm curl} A(x) =  \frac{12}{(1+|x|^2)^3} \left( (1-|x|^2) w +2 (x \cdot w) x + 2 x \wedge w \right) \ .
\end{equation}
It is amusing to note that the field lines of $A$ as well as $B$ wind around a family of nested tori and are in fact
 the great circles associated with the Hopf-fibration on $\Sp^3$ pulled back to $\R^3$ by the stereographic projection.
 
A considerable amount of research has gone into finding more examples and trying to understand the structure of the zero mode equation. Erd\"os and Solovej \cite{ES} showed that the above example is a special case of a class of fields that emerge by pulling back the two dimensional zero modes of Aharonov-Casher \cite {AC} using the Hopf map.  Other examples where given by Adam, Muratori and Nash \cite{AMN} in connection with the importance of zero modes in QED for the understanding of anomalies. Further examples were found also by Elton \cite{E}, by Saito and Umeda \cite{saitoumeda1} and more recently by Ross-Schroers \cite{RS}.

In contrast to the two dimensional problem solved in \cite{AC}  there does not seem to be a particular geometric structure that allows to classify all the zero modes. This was shown by Balinsky-Evans in \cite{BE}
where they show that the magnetic fields that do {\it not } support zero modes form an open dense set in $L^{3/2}(\R^3:\R^3)$. They also showed that  zero modes disappear by varying the size of the magnetic field; there is at most a discrete set of $t$ values such that $tB$ supports a zero mode.
Thus, it is in general difficult to say whether a certain magnetic field supports a zero mode or not. The problem of absence of zero modes was analyzed by Kalf, Okaji and Yamada \cite{Kalfetal1} and \cite{Kalfetal2}. The authors give sharp pointwise conditions on the decay rate of the vector potential that guarantees absence of zero modes. Other results in this direction can be found in \cite{saitoumeda2}. We mention in this context \cite{Cossettietal} where absence of zero modes (in fact, absence of any eigenvalues) was proved, however, under a certain implicit smallness condition on the magnetic field. For further results on the absence of positive eigenvalues for Schr\"odinger as well as Dirac operators see \cite{hundertmarketal}.

In this paper we take a different route and try to address the question of absence of zero modes in a quantitative fashion.
One traditional measure is the size of
$$
\Vert B \Vert_{3/2} =\left( \int_{\R^3} |B(x)|^{3/2} dx \right)^{2/3}
$$
which is invariant under scaling, in contrast to the field energy. It is straightforward to see that if 
$\Vert B \Vert_{3/2}$ is too small, then there cannot be zero modes. Indeed, formally squaring the Dirac operator $\sigma \cdot (-i \nabla -A)$ yields the Schr\"odinger equation
\begin{equation}\label{energy}
(-i\nabla -A)^2 \psi - \sigma \cdot B \psi = 0  \ .
\end{equation}
The associated quadratic form is given by 
$$
\Vert (-i\nabla - A)\psi\Vert_2^2 -\int_{\R^3} \langle \psi, \sigma \cdot B \psi\rangle dx
\ge \Vert (-i\nabla - A)\psi\Vert^2 -\int_{\R^3} | \psi|^2 |B(x)|  dx
$$
which in turn, using the diamagnetic inequality and H\"older's inequality, is bounded below by
$$
\Vert \nabla |\psi|\Vert_2^2 - \Vert B \Vert_{3/2} \Vert \psi \Vert_6^2 \ .
$$
Sobolev's inequality states that
$$
\Vert \nabla |\psi|\Vert_2^2 \ge S_3 \Vert \psi \Vert_6^2
$$
where 
$$
S_3 = \frac{3}{4} |\Sp^3|^{2/3}  = \frac{3}{4} \left(2\pi^2\right) ^{2/3} = 3 \left(\frac{\pi}{2}\right)^{4/3} \ .
$$
Thus, we see that for  $\Vert B\Vert_{3/2} < S_3$ there cannot be any zero modes.

Needless to say that the above argument is much too rough, in particular the spinor structure is completely ignored. The aim of this paper is to explore the particular role played by the spinor which allows us to give an improvement. 
Our first result is

\begin{theorem}\label{magnetic} 
Let $B \in L^{3/2}(\R^3: \R^3)$ be a magnetic field, i.e., ${\rm div}\, B=0$. If \eqref{zeromode} has a weak solution $0\not\equiv \psi \in L^p(\R^3: \C^2)$ for some $3/2<p<\infty$, then
$$
\Vert B\Vert_{3/2} \ge  2\,S_3 \ .
$$
\end{theorem}
It is an open question whether the estimate in Theorem \ref{magnetic} is sharp or not. If one considers
magnetic fields that are in $L^{3/2}_w$, then we provide, under suitable regularity assumptions on the spinor $\psi$ an estimate on the  $L^{3/2}_w$-norm of the magnetic field that is sharp in the sense that it is
saturated for a magnetic monopole field. Of course, the monopole field is outside the class of fields we consider but it provides an informative example as to the inner workings of the various inequalities used in the proof of Theorem \ref{magnetic}. This will be presented in Section \ref{optimality}.

In \cite{LossYau} it was observed that the spinor \eqref{lossyau1} solves the equation
$$
-i \sigma \cdot \nabla \psi = \frac{3}{1+|x|^2} \psi \ ,
$$
and by setting 
$$
A = \frac{3}{1+|x|^2} \frac{\langle \psi, \sigma \psi\rangle}{\langle \psi, \psi \rangle}
$$
one obtains \eqref{lossyau2}. The basis for this observation is the identity
\begin{equation} \label{spinidentity}
\sigma \cdot \frac{\langle \psi, \sigma \psi\rangle }{\langle \psi, \psi \rangle} \psi = \psi \ .
\end{equation}
It is therefore a reasonable question to consider the equation
$$
-i \sigma \cdot \nabla \psi =3\lambda(x)\psi
$$
and to try to find necessary conditions on $\lambda$ for a solution to exist.
In a first attempt one could consider 
\begin{align*}
	\int_{\R^3} |\nabla \psi|^2 dx & = \int_{\R^3} \langle -i\sigma\cdot \nabla \psi, -i\sigma \cdot \nabla \psi\rangle dx
	= 9 \int_{\R^3} |\lambda|^2 |\psi|^2 dx \\
	& \le 9 \left(\int_{\R^3} |\lambda|^3 dx\right)^{2/3} \left(\int_{\R^3} |\psi|^6 dx\right)^{1/3} \ .
\end{align*}
Using Sobolev's inequality one obtains the bound
$$
\left(\int_{\R^3} |\lambda|^3 dx\right)^{2/3} \ge
\frac{1}{9}\,S_3 \ .
$$
We can do substantially better than that.
We have
\begin{theorem}\label{spinor}
Assume that $\lambda \in L^3(\R^3)$ is a real function. If the equation 
$$
-i \sigma \cdot \nabla \psi = 3\lambda(x) \psi
$$
has a weak solution $0\not\equiv\psi \in L^p(\R^3: \C^2)$ for some $3/2<p<\infty$, then
\begin{equation} \label{sharpspinor}
\left( \int_{\R^3} |\lambda|^3 dx \right)^{2/3} \ge \frac{1}{3}\,S_3 = \frac14\, |\Sp^3|^{2/3} \ .
\end{equation}
There is equality in \eqref{sharpspinor} if 
$$
\lambda(x) = \frac{1}{1+|x|^2} \ ,
$$
in which case \eqref{lossyau1} is a solution. 
\end{theorem}
This result is very closely related to an inequality of Hijazi for eigenvalues of Dirac operators in conformal geometry. We discuss this connection in detail in Section \ref{hijaziapproach}. 

In a very informative paper, Dunne and Min \cite{Dunne-Min} obtained  the zero modes given by \eqref{lossyau1} and \eqref{lossyau2}  as zero modes of a Dirac equation on $\Sp^3$ with a field of fixed helicity pulled back via stereographic projection. This picture allowed them to obtain examples of zero modes in any odd dimensions $d$. In arbitrary (even or odd) dimension $d$, the dimension of the spinors on $\R^d$ is $N=2^\nu$, where $\nu$ is related to $d$ by $d=2\nu+1$ if $d$ is odd and $d=2\nu$ if $d$ is even. The Dirac matrices $\gamma_j$, $j=1, \dots, d$, satisfy
$$
\gamma_i \gamma_j + \gamma_j \gamma_i = 2 \delta_{i,j} \ ,
\qquad 1 \le i,j \le d \ .
$$
In odd dimensions, the zero mode spinors of Dunne and Min are of the form
\begin{equation}
	\label{eq:dunneminspinor}
	\psi(x) = \frac{1 + i \gamma \cdot x}{(1+|x|^2)^{\frac{d}{2}}} \left[\begin{array}{c}  \phi_0 \\ 0\end{array}\right]
\end{equation}
where $\left[\begin{array}{c} \phi_0 \\  0\end{array}\right]$ is a well-chosen constant normalized spinor of dimension $N= 2^\nu= 2^{\frac{d-1}{2}}$. When $d=3$ this reduces to the previous case \eqref{lossyau1}. The vector potential
is of the form
$$
A_i(x) =\frac{d }{(1+|x|^2)^2} \left[ -2 \sum_{j=1}^d J_{ij} x_j + 2x_ix_{d}\right] \ ,\qquad i=1,2, \dots d-1 \ ,
$$
where $J= {\rm diag}(i\sigma_2, \dots, i\sigma_2, -i\sigma_2)$ and
$$
A_{d}(x) = d \left(\frac{1- |x|^2 +2 x_{d}^2}{(1+|x|^2)^2} \right) \ .
$$
Again, when $d=3$ this reduces to the case \eqref{lossyau2}. We note that this construction
seems to work only in odd dimension. In Appendix \ref{dirac} we give a more elementary derivation of the solutions found by Dunne and Min.

The existence
of such type of zero modes suggests generalizations of  Theorem \ref{magnetic} and Theorem \ref{spinor} to higher dimensions. We emphasize that, in contrast to the Dunne-Min construction, we do \emph{not} assume that $d$ is odd.

\begin{theorem} \label{genmagnetic}
	Let $d\geq 3$. If the equation
$$
\gamma \cdot (-i\nabla -A(x))\psi = 0
$$
has a weak solution $0\not\equiv \psi\in L^p(\R^d:\C^{2^\nu})$ for some $\frac{d}{d-1}<p<\infty$, then
$$
\left(\int_{\R^d} |B(x)|^{d/2} dx\right)^{2/d} \ge \nu^{-1/2} \frac{d-1}{d-2}\, S_d \ ,
$$
where $\nu=(d-1)/2$ if $d$ is odd and $\nu=d/2$ if $d$ is even, where
$$
|B(x)| = \left( \sum_{j<k} |\partial_j A_k(x) - \partial_k A_j(x)|^2 \right)^{1/2}
$$
and where $S_d = \frac{d(d-2)}{4} |\Sp_d|^{2/d}$ is the Sobolev constant.
\end{theorem}

For $d\ge 3$ consider the spinor
\begin{equation}
	\label{eq:spinorgeneraldim}
	\psi = \frac{1+i\gamma \cdot x}{(1+|x|^2)^{\frac{d}{2}}}\ \Phi_0
\end{equation}
where $\Phi_0$ is any constant spinor. 
An elementary computation shows that
\begin{equation} \label{spinoroddeven}
	-i \gamma \cdot \nabla \psi = \frac{d}{1+|x|^2} \psi \ .
\end{equation}

\begin{theorem}\label{spinorgeneral}
	Let $d\geq 3$ and assume that $\lambda\in L^d(\R^d)$ is a real function. If the equation
	\begin{equation}\label{generalequation}
		-i\gamma \cdot \nabla \psi = d\, \lambda(x)\, \psi
	\end{equation}
	has a weak solution $0\not\equiv \psi\in L^p(\R^d:\C^{2^\nu})$ for some $\frac{d}{d-1}<p<\infty$, then
$$
\left( \int_{\R^d} |\lambda(x)|^d  dx\right)^{\frac{2}{d} }\ge \frac{1}{d(d-2)}\, S_d = \frac14\, |\Sp^d|^{2/d} \ .
$$
Moreover, there is equality  if $\lambda(x) = \frac{1}{1+|x|^2}$, in which case \eqref{eq:dunneminspinor} is a solution.
\end{theorem}

The main difference between the proofs of Theorems \ref{magnetic} and \ref{spinor} and those of Theorems \ref{genmagnetic} and \ref{spinorgeneral} is of algebraic nature; the Gamma matrices in higher dimensions are a bit more difficult to handle.

As an aside, we mention that, according to \eqref{spinoroddeven}, the spinor \eqref{eq:spinorgeneraldim} satisfies the nonlinear Dirac equation
$$
-i \gamma\cdot\nabla \psi = d\,|\psi|^{\frac{2}{d-1}} \psi \ .
$$
Solutions to this equation have been considered in \cite{Borelli--Frank} and \cite{Borrellietal}.  

Returning to the stability problem and recalling the critical nuclear charge $Z_c$ in \eqref{criticalzee}, in \cite{FLL} it was shown that $Z_c$ is bounded below by $24.0/8\pi \alpha^2$ and the results so far suggest that some improvement should be possible. This is indeed true albeit not by very much.
\begin{theorem}\label{improvedz}
The critical charge  $Z_c8\pi \alpha^2$ is bounded below by $32\pi/3 \approx 33.51$
\end{theorem}
This leads to a numerical bound $Z_c \ge  \frac43 \frac{1}{\alpha^2} \approx 25,025$ which is slightly better than the
$17,900$ given in \cite{FLL} but still much worse than the upper bound $208,000$ given there. 

The main tool that leads to these results is an improved `diamagnetic' inequality. The diamagnetic inequality
states that $|\nabla |\psi||^2 \le |(-i\nabla -A)\psi|^2$. It turns out that for zero modes, the stronger inequality
 $|\nabla |\psi||^2 \le \frac23|(-i\nabla -A)\psi|^2$ holds (see Lemma \ref{diamagref}). All the other results follow
 from arguments that are variations on this theme. 
In our presentation we start first with a regularity theorem for spinors and then continue with the proofs of the stated theorems.  We end the paper with a number of open problems.

\subsection*{Acknowledgements}
The authors would like to thank H.~Kovarik and M.~Lewin for helpful remarks. Partial support through U.S. National Science Foundation grants DMS-1363432 and DMS-1954995 (R.L.F.) and DMS-1856645 (M.L.) and through  through the Deutsche Forschungsgemeinschaft (German Research Foundation) through Germany’s Excellence Strategy EXC-2111-390814868 (R.L.F.) is acknowledged.


\section{Regularity of zero modes}

Throughout this section, we assume that $A \in L^3(\R^3:\R^3)$. As we will see in the next section, for given $B\in L^{3/2}(\R^3:\R^3)$ with ${\rm div} B =0$ one can always find an $A \in L^3(\R^3:\R^3)$ with ${\rm curl} A = B$. If $\psi\in L^p(\R^3:\C^2)$ for some $p\geq 3/2$, then $\sigma\cdot A\psi$ is locally integrable and therefore it makes sense to consider the equation
$$
\sigma\cdot (-i\nabla-A)\psi = 0
\qquad\text{in}\ \R^3
$$
interpreted in the sense of distributions. We now prove regularity of solutions to this equation.

\begin{theorem} \label{regularity}
Fix $3/2<p<\infty$ and let $\psi \in L^p(\R^3: \C^2)$ be a solution of the zero mode equation \eqref{zeromode}. Then $\psi\in L^r(\R^3:\C^2)$ for any $3/2< r<\infty$.
\end{theorem}

Our proof of this theorem extends an argument from \cite{Borelli--Frank} whose roots can be traced back to \cite{Jannelli--Solimini}.

For the proof we recall (see, e.g., \cite{LL}) that the Hardy-Littlewood-Sobolev (HLS) inequality states that
$$
\Big | \int_{\R^n} \int_{\R^n} \frac{f(x) g(y)}{|x-y|^\lambda} dx dy \Big | \le C_{q,r} \Vert f \Vert_q \Vert g \Vert_r \ , \quad 
\frac1q +\frac1r + \frac\lambda n = 2 \ ,
$$
and, importantly, $q,r>1$. The latter inequalities are the reason for the assumption $p>3/2$ in Theorem \ref{regularity}.

\begin{proof}
The Green's function for the Dirac operator is
$$
\Gamma(x-y) = \frac{i}{4\pi} \frac{\sigma\cdot(x-y)}{|x-y|^3} \,.
$$
We claim that the equation for $\psi$ can be rewritten as
$$
\psi = \Gamma * (\sigma\cdot A\psi) \,.
$$
Indeed, let $\tilde\psi:= \Gamma * (\sigma\cdot A\psi) $. By the assumptions on $\psi$ and $A$ we have $\sigma\cdot A\psi \in L^q$ with $1/q= 1/3+ 1/p$ and therefore, by HLS, $\tilde\psi\in L^p$. Moreover, one can verify that $\sigma\cdot(-i\nabla)(\psi-\tilde\psi)=0$ in $\R^3$ in the sense of distributions.  Hence, 
$-\Delta (\psi-\tilde \psi) = (\sigma\cdot(-i\nabla))^2(\psi-\tilde\psi)=0$ and thus $\psi=\tilde\psi$ is a harmonic distribution. By standard elliptic regularity theory (or a distributional version of Weyl's Lemma \cite{Weyl}), $\psi-\tilde \psi$ is smooth. Note that by H\"older's inequality $ \int_{B_r(a)} |\psi -\tilde \psi| dx \le |B_r(a)|^{1/p'} \Vert \psi - \tilde \psi \Vert_p$, $1/p+1/p'=1$,  which implies that the average of $\psi -\tilde \psi$ over the ball $B_r(a)$ can be made arbitrarily small by choosing $r$ large. The mean value property of harmonic functions then implies that $\psi-\tilde \psi $ vanishes identically. This yields the claimed integral equation.
 
 Throughout the remainder of this proof we fix a parameter $r$ satisfying $3/2<r\leq \frac{3p}{3-p}$ if $3/2<p<3$ and $3/2<r<\infty$ if $p\geq 3$. For $M>0$, let
$$
S_M:= \sup\left\{ \left| \int_{\R^3}\langle  \phi, \psi \rangle \,dx \right| :\ \|\phi\|_{r'}\leq 1\,,\ \|\phi\|_{p'}\leq M \right\},
$$
where, again, $p'$ is the index dual to $p$, i.e., $1/p+1/p'=1$.
The fact that $\phi\in L^{p'}$ guarantees that $S_M<\infty$ for any $M$ (in fact, $S_M \leq M\,\|\psi\|_p $). In the following we will show that there is a constant $C<\infty$ depending only on $A$ such that
$$
S_M \leq C \|\psi\|_p 
\qquad\text{for all}\ M>0 \,.
$$
By density of $L^{r'}\cap L^{p'}$ in $L^{r'}$ and duality, this bound implies that $\psi\in L^r$ with $\|\psi\|_r\leq C\|\psi\|_p$.

To prove the above bound, let $\varepsilon>0$ be a parameter, which will later be fixed depending on $A$. Clearly, we can decompose
$$
\sigma\cdot A = F_\varepsilon +G_\varepsilon
$$
where $F_\varepsilon$ and $G_\varepsilon$ are functions on $\R^3$ taking values in the Hermitian $2\times 2$ matrices such that $\|G_\varepsilon\|_3\leq\varepsilon$ and $F_\varepsilon$ is bounded and has compact support. The integral equation gives
\begin{equation} \label{integralequ}
\psi = \Gamma* (F_\varepsilon\psi) +\Gamma*(G_\varepsilon\psi) \,.
\end{equation}
Let $\phi\in L^{r'}\cap L^{p'}(\R^3:\C^2)$ with $\|\phi\|_{r'}\leq 1$ and $\|\phi\|_{p'}\leq M$
and set
$$
\chi_\varepsilon := G_\varepsilon(\Gamma*\phi) \,.
$$
We show momentarily that $\chi_\varepsilon\in L^{p'}$. This justifies that we can integrate \eqref{integralequ} against $\phi$ and obtain
\begin{equation}
\label{eq:decomp}
\int_{\R^3}\langle \phi ,\psi\rangle \,dx = \int_{\R^3} \langle \phi, \Gamma* (F_\varepsilon\psi)  \rangle\,dx + \int_{\R^3}\langle  \chi_\varepsilon, \psi\rangle \,dx \,.
\end{equation}
We estimate the two terms on the right side separately. We introduce the parameter $1 <s\leq p$ by $1/s = 1/r + 1/3$. Note that since $r \leq \frac{3p}{3-p}$ for $p<3$ we have indeed that $s\leq p$. Moreover, $s>1$ follows from $r>3/2$. For the first term, we have
\begin{align*}
\left| \int_{\R^3} \langle \phi , \Gamma* (F_\varepsilon\psi)\rangle  \,dx \right| \leq \|\Gamma* (F_\varepsilon\psi) \|_r \underset{\mathrm{HLS}}{\lesssim} \|F_\varepsilon\psi\|_s \leq  \|F_\varepsilon\|_\frac{ps}{p-s} \|\psi\|_p \,,
\end{align*}
where we interpret $\frac{ps}{p-s}=\infty$ if $p=s$. We now turn to the second term in \eqref{eq:decomp} and begin by showing that $\chi_\varepsilon\in L^{p'}$, which was needed in the above computation. Indeed, with $1/p' =1/3+1/t$,
\begin{align*}
\left\| \chi_\varepsilon \right\|_{p'} & \leq \|G_\varepsilon\|_3\|\Gamma*\phi\|_t \underset{\mathrm{HLS}}{\lesssim}  \|G_\varepsilon\|_3\|\phi\|_{p'} \leq \|G_\varepsilon\|_3 M \,.
\end{align*}
A second bound on the same term is
\begin{align*}
\left\| \chi_\varepsilon \right\|_{r'} & \leq \|G_\varepsilon\|_3\|\Gamma*\phi\|_{s'} \underset{\mathrm{HLS}}{\lesssim}  \|G_\varepsilon\|_3\|\phi\|_{r'} \leq  \|G_\varepsilon\|_3 \,.
\end{align*}
The previous two bounds show that there is a universal constant $C_1$ such that
$$
\tilde\phi := \chi_\varepsilon/(C_1 \|G_\varepsilon\|_3)
$$
satisfies $\|\tilde\phi\|_{r'}\leq 1$ and $\|\tilde\phi\|_{p'}\leq M$. Thus, by definition of $S_M$,
$$
\left| \int_{\R^3}\langle  \tilde\phi, \psi \rangle\,dx \right| \leq S_M \,,
$$
which is the same as
$$
\left| \int_{\R^3}\langle \chi_\varepsilon, \psi\rangle \,dx \right| \leq S_M C_1 \|G_\varepsilon\|_3 \,.
$$
This is the desired bound on the second term in \eqref{eq:decomp}.

Combining this with the bound on the first term we conclude that
$$
\left| \int_{\R^3} \langle \phi, \psi \rangle\,dx \right| \leq C_2 \|F_\varepsilon\|_\frac{ps}{p-s} \|\psi\|_p + S_M C_1 \|G_\varepsilon\|_3 \,.
$$
Taking the supremum over all $\phi\in L^{r'}\cap L^{p'}$ with $\|\phi\|_{r'}\leq 1$ and $\|\phi\|_{p'}\leq M$ we obtain
$$
S_M \leq C_2 \|F_\varepsilon\|_\frac{ps}{p-s} \|\psi\|_p+ S_M C_1 \|G_\varepsilon\|_3 \,.
$$
We now recall that $\|G_\varepsilon\|_3\leq \varepsilon$. Therefore, choosing $\varepsilon = (2C_1)^{-1}$ and recalling that $S_M<\infty$, we obtain
$$
S_M \leq 2 C_2 \|F_\varepsilon\|_\frac{ps}{p-s} \|\psi\|_p \,,  \qquad s= \frac{3r}{3+r} \ .
$$
This is the claimed bound.

To summarize, we have shown that for all $3/2 < r \leq \frac{3p}{3-p}$ if $3/2<p<3$ and for all $3/2<r<\infty$ if $p\geq 3$ we have
$$
\Vert \psi \Vert_r \le C \Vert \psi \Vert_p \,,
$$
where $C$ is a constant that only depends on  $A, p$ and $r$. In case $p\geq 3$ this is the claimed result.
In case $3/2<p<3$ we have shown, in particular, that $\psi \in L^\frac{3p}{3-p} (\R^3: \C^2)$. Since $\frac{3p}{3-p} > 3$ we can repeat the argument and obtain the claimed result.
\end{proof}
These regularity estimates allow us to improve on a result by Balinsky--Evans  \cite{Balinsky-Evans-Lewis} and Benguria--van den Bosch \cite{BenguriaBosch}.
Let $B\in L^{3/2}(\R^3:\R^3)$ with ${\rm div}\, B=0$ and define $A$ by \eqref{eq:defa} below, so that $A\in L^3(\R^3:\R^3)$ and ${\rm curl}\, A=B$. Following Balinsky, Evans and Lewis \cite{Balinsky-Evans-Lewis} we consider the operator
$$
S := |B|^{1/2} ( (\sigma\cdot(-i\nabla-A))^2+|B|)^{-1/2}
\qquad\text{in}\ L^2(\R^3:\C^2) \,.
$$
Note that the kernel of the operator $ (\sigma\cdot(-i\nabla-A))^2+|B|$ is trivial, so the operator $( (\sigma\cdot(-i\nabla-A))^2+|B|)^{-1/2}$ is densely defined. Using the diamagnetic inequality and the fact that $ (\sigma\cdot(-i\nabla-A))^2+|B| \geq (-i\nabla +A)^2$, it is not difficult to see that $S$ is a bounded operator in $L^2(\R^3:\C^2)$. We set
$$
\delta(B) := \| 1-S^*S \| \,.
$$
By gauge invariance, it is easy to see that the right side, indeed, only depends on $B$ and not on $A$.

\begin{theorem}\label{benguriabosch}
	Let $B\in L^{3/2}(\R^3:\R^3)$. Then $\delta(B)=0$ if and only if $(\sigma\cdot(-i\nabla-A))^2$ has a zero mode in $L^2(\R^3:\C^2)$.
\end{theorem}

The fact that if $(\sigma\cdot(-i\nabla-A))^2$ has a zero mode, then $\delta(B)=0$ was shown by Balinsky, Evans and Lewis \cite{Balinsky-Evans-Lewis}. The converse implication was shown by Benguria and van den Bosch
\cite{BenguriaBosch} under an additional pointwise decay condition on $B$. Our contribution here is to note that this additional pointwise decay condition is not necessary. We thank H.~Kovarik for drawing our attention to this question.

Indeed, if $\delta(B)=0$, then Benguria and van den Bosch showed (see Lemma 3.1 in \cite{BenguriaBosch} ) that there is a $\psi\in L^6(\R^3:\C^2)$ such that $\sigma\cdot(-i\nabla-A)\psi=0$ in $\R^3$. Then they use the decay assumption on $B$ to deduce that $\psi\in L^2(\R^3:\C^2)$. The same conclusion, however, follows from our Theorem \ref{regularity} with $p=6$ and $r=2$, without any additional assumption. This proves Theorem \ref{benguriabosch} as stated.


\section{Proof of Theorem \ref{magnetic}}

Given $B\in L^{3/2}(\R^3:\R^3)$ with ${\rm div} B=0$, our choice for the vector potential is
\begin{equation}
\label{eq:defa}
A(x)=-\frac{1}{4\pi} \int_{\R^3} \frac{x-y}{|x-y|^3} \wedge B(y) dy \ ,
\end{equation}
where $a\wedge b$ denotes the cross product of two vectors. By the Hardy-Littlewood-Sobolev inequality, $A \in L^3(\R^3:\R^3)$. Although not important for our purpose, we note that ${\rm div} A = 0$. By Theorem \ref{regularity}, the spinor $\psi$ is in any $L^r$-space for $3/2< r < \infty$ and, in particular, it is in $L^6(\R^3: \C^2)$.
Since $|A||\psi| \in L^2(\R^3)$ we find from \eqref{zeromode} that $|-i\sigma \cdot \nabla \psi| \in L^2(\R^3)$. The formula $\int_{\R^3} |\nabla \psi|^2 dx = \int_{\R^3} |-i \sigma \cdot \nabla \psi|^2 dx $ shows that $\psi \in \dot H^1(\R^3: \C^2)$. In fact, $\psi\in H^1(\R^3:\C^2)$, since $\psi$ is also squaresummable.

For the proof we need the following improvement over the diamagnetic inequality.

\begin{lemma}\label{diamagref}
Let $\psi \in L^p(\R^3:\C^2)$, $3/2<p < \infty$, satisfy $\sigma\cdot(-i\nabla-A)\psi =0$. Then $\psi \in H^1(\R^3:\C^2)$,
and $|\psi| \in H^1(\R^3)$ as well and moreover,  almost everywhere in $\R^3$,
$$
\left|\nabla |\psi|\right|^2 \leq \frac23 \left|(-i\nabla-A)\psi \right|^2 \,.
$$
\end{lemma}

An inequality of this form appears in \cite{Fe}, but only for smooth $\psi$ and away from the zero set of $\psi$; see also \cite{CaGaHe}. One can use unique continuation results (see \cite{Estebanetal} and the references therein) to show that the zero set of $\psi$ has measure zero and thereby obtain the inequality almost everywhere. (We thank M. Lewin for this remark.) We choose a different and technically simpler path, which combines the arguments in \cite{CaGaHe} with the chain rule for Sobolev functions. Throughout the following, the functions
$$
|\psi|_\varepsilon = \sqrt{|\psi|^2 +\varepsilon^2}
$$
for $\varepsilon>0$ will play an important role.

\begin{proof} 
	In the discussion before the theorem we have already shown that $\psi \in H^1(\R^3:\C^2)$. It is well-known (see, e.g., \cite[Theorem 6.17]{LL}) that this implies $|\psi|\in H^1(\R^3:\C^2)$. For the function $|\psi|_\epsilon$ introduced above	we have (see, for instance,  \cite[Theorem 6.16]{LL})
$$
\partial_j |\psi|_\varepsilon  = \re\left\langle \frac{\psi}{|\psi|_\varepsilon}, \partial_j \psi \right\rangle = \re\left\langle \frac{\psi}{|\psi|_\varepsilon}, \left(\partial_j-iA_j\right) \psi \right\rangle,
$$
and therefore
\begin{equation} \label{magic1}
\left| \nabla |\psi|_\varepsilon \right| = \frac{\nabla |\psi|_\varepsilon}{\left|\nabla |\psi|_\varepsilon\right|}\cdot \nabla |\psi|_\varepsilon  = \re\left\langle \frac{\psi}{|\psi|_\varepsilon} \frac{\nabla |\psi|_\varepsilon}{\left|\nabla |\psi|_\varepsilon\right|}, \left(\nabla-iA\right) \psi \right\rangle \ .
\end{equation}
Here on the right side, for each fixed $x\in\R^3$ we consider $\left(\nabla-iA\right) \psi$ as an element of $\R^3\otimes\C^2$. (The index coming from $\R^3$ labels the component of the gradient, and the index coming from $\C^2$ labels the components of the spinor.) On $\R^3\otimes\C^2$ we introduce the projection 
\begin{equation} \label{projection}
\left( \Pi (\alpha\otimes v)\right)_j = \alpha_j v - \frac{1}{3} \sigma_j \sigma\cdot\alpha v
\qquad\text{for}\ j=1,2,3 \,.
\end{equation}
where $\alpha\in\R^3$ and $v\in\C^2$.
Since $\sigma\cdot(-i\nabla -A)\psi =0$, we have
$$
\left(\nabla-iA\right) \psi = \Pi \left(\nabla-iA\right) \psi \,,
$$
and therefore, since $\Pi$ is self-adjoint in $\R^3\otimes\C^2$ with respect to the inner product
$(\alpha \otimes v, \beta \otimes w) =\sum_j \alpha_j \beta_j \langle v, w \rangle$,
$$
\left| \nabla |\psi|_\varepsilon \right| = \re \left\langle \Pi \left( \frac{\psi}{|\psi|_\varepsilon} \frac{\nabla |\psi|_\varepsilon}{\left|\nabla |\psi|_\varepsilon\right|} \right), \left(\nabla-iA\right) \psi \right\rangle \,.
$$
We now bound
$$
\left| \nabla |\psi|_\varepsilon \right| \leq \left| \Pi \left( \frac{\psi}{|\psi|_\varepsilon} \frac{\nabla |\psi|_\varepsilon}{\left|\nabla |\psi|_\varepsilon\right|} \right) \right| \left| \left(\nabla-iA\right) \psi \right| \ .
$$
A simple computation shows that
$$
\left| \Pi (\alpha\otimes v) \right|^2 = \frac{2}{3} |\alpha|^2 |v|^2 \,.
$$
This identity with $v=\psi/|\psi|_\varepsilon$ and $\alpha =\nabla |\psi|_\varepsilon/|\nabla |\psi|_\varepsilon|$,
and by letting $\varepsilon$ tend to zero, yields the claimed inequality.
\end{proof}
\begin{lemma}\label{squareroot}
	For any $\psi\in H^1(\R^3: \C^2)$ and any $\varepsilon>0$, the function $|\psi|_\varepsilon^{1/2}$ is weakly differentiable with $\nabla |\psi|_\varepsilon^{1/2}\in L^2(\R^3)$ and one has almost everywhere and in the sense of $L^1$
	\begin{equation} \label{formula1}
	\left|\nabla |\psi|_\varepsilon^{1/2}\right|^2 = \frac12 \left( \re \left\langle \nabla \frac{\psi}{|\psi|_\varepsilon},\nabla\psi \right\rangle - |\psi|_\varepsilon^{-1} |\nabla\psi|^2 \right) + \frac{3}{4} \frac{|\psi|^2}{|\psi|_\varepsilon^3} \left|\nabla|\psi|\right|^2 \,.
	\end{equation}
\end{lemma}
\begin{proof}
	By the chain rule for Sobolev functions (see for instance, Theorem 6.16 in \cite{LL}), the function $|\psi|_\varepsilon^{1/2}$ is weakly differentiable and
	$$
	\nabla |\psi|_\varepsilon^{1/2} = \frac12 |\psi|_\varepsilon^{-3/2} \re \left\langle \psi,\nabla\psi \right\rangle \,.
	$$
	Since $|\psi||\psi|_\varepsilon^{-3/2}\leq \varepsilon^{-1/2}$, we have $\nabla |\psi|_\varepsilon^{1/2}\in L^2(\R^3)$.
	
	Using the above formula for the gradient of $|\psi|_\varepsilon^{1/2}$ with $\re \left\langle\psi, \nabla\psi\right\rangle = |\psi|\nabla|\psi|$ on the left side of \eqref{formula1} we can rewrite the assertion of the lemma as
	\begin{equation}
	\label{eq:squarerootproof}
	0 = \re \left\langle \nabla \frac{\psi}{|\psi|_\varepsilon},\nabla\psi \right\rangle - |\psi|_\varepsilon^{-1} |\nabla\psi|^2 + \frac{|\psi|^2}{|\psi|_\varepsilon^3} \left|\nabla|\psi|\right|^2 \,.
	\end{equation}
	Again, using the chain rule, we see that the function $|\psi|_\varepsilon^{-1}$ is weakly differentiable with $\nabla |\psi|_\varepsilon^{-1} = - |\psi|_\varepsilon^{-3} |\psi|\nabla|\psi|$. Therefore, by the product rule for weak derivatives, $\psi |\psi|_\varepsilon^{-1}$ is weakly differentiable with
\begin{equation} \label{badterm}
\nabla \frac{\psi}{|\psi|_\varepsilon} = \frac{\nabla \psi}{|\psi|_\varepsilon} - \frac{\psi |\psi|\nabla|\psi|}{|\psi|_\varepsilon^3} \,.
\end{equation}
Thus,
$$
\left\langle \nabla \frac{\psi}{|\psi|_\varepsilon},\nabla\psi \right\rangle = \frac{|\nabla \psi|^2}{|\psi|_\varepsilon} - \frac{\left\langle\psi,\nabla\psi\right\rangle \cdot |\psi|\nabla|\psi|}{|\psi|_\varepsilon^3} \ .
$$	
Using the fact that $\re \left\langle\psi,\nabla\psi\right\rangle = |\psi|\nabla|\psi|$, we obtain \eqref{eq:squarerootproof}, which proves the lemma.
\end{proof}
\begin{remark} \label{psiepsrem}
	If $\psi\in H^1(\R^3:\C^2)$, then $\frac{\psi}{|\psi|_\varepsilon}  \in H^1(\R^3:\C^2)$ for any $\varepsilon>0$. Indeed, using \eqref{badterm} one easily sees that
\begin{equation*}
\Big |\nabla \frac{\psi}{|\psi|_\varepsilon} \Big |\le \frac1\varepsilon [|\nabla \psi|+|\nabla|\psi ||].
\end{equation*}
\end{remark}

The following lemma is standard and easy to prove using an approximation argument.
\begin{lemma} \label{byparts}
	For any $\psi,\eta\in H^1(\R^3:\C^2)$, one has
	$$
	\int_{\R^3} \left\langle (\nabla-iA) \eta,(\nabla-iA)\psi \right\rangle dx 
	= \int_{\R^3} \left( \left\langle \sigma\cdot(\nabla-iA) \eta,\sigma\cdot(\nabla-iA)\psi \right\rangle
	+ \left\langle\eta,\sigma\cdot B\psi\right\rangle\right) dx \,.
	$$
\end{lemma}
\begin{proof}[Proof of Theorem \ref{magnetic}]
Lemma \ref{squareroot} and simple computations show that
	\begin{align*}
	\left|\nabla |\psi|_\varepsilon^{1/2}\right|^2 & = \frac12 \re \left\langle (\nabla-iA) \frac{\psi}{|\psi|_\varepsilon},(\nabla-iA)\psi \right\rangle - \frac{1}{4|\psi|_\varepsilon}\left( 2 \left|(\nabla-iA)\psi\right|^2 - 3 \frac{|\psi|^2}{|\psi|_\varepsilon^2} \left|\nabla|\psi|\right|^2 \right) \\
	& = \frac12 \re \left\langle (\nabla-iA) \frac{\psi}{|\psi|_\varepsilon},(\nabla-iA)\psi \right\rangle - \frac{1}{4|\psi|_\varepsilon}\left( 2 \left|(\nabla-iA)\psi\right|^2 - 3 \left|\nabla|\psi|\right|^2 \right) \\
	& \quad - \frac{3\varepsilon^2}{4|\psi|_\varepsilon^3} \left|\nabla|\psi|\right|^2 \,.
	\end{align*}
Lemma \ref{diamagref} then yields the inequality
$$
\left|\nabla |\psi|_\varepsilon^{1/2}\right|^2 \le  \frac12 \re \left\langle (\nabla-iA) \frac{\psi}{|\psi|_\varepsilon},(\nabla-iA)\psi \right\rangle - \frac{3\varepsilon^2}{4|\psi|_\varepsilon^3} \left|\nabla|\psi|\right|^2 \ .
$$
Since $\frac{\psi}{|\psi|_\varepsilon} \in H^1(\R^3:\C^2)$ by Remark \ref{psiepsrem}, we may integrate this expression and, using Lemma \ref{byparts} as well as the the zero mode equation \eqref{zeromode}, we arrive at
\begin{equation} \label{crux}
\int_{\R^3} \left|\nabla |\psi|_\varepsilon^{1/2}\right|^2\,dx \le \frac12 \int_{\R^3} \left\langle \frac{\psi}{|\psi|_\varepsilon},\sigma\cdot B\psi \right\rangle dx  - \frac{3\varepsilon^2}{4} \int_{\R^3} \frac{|\nabla|\psi||^2}{|\psi|_\varepsilon^3}\,dx \ .
\end{equation}
The left side we bound from below, using Sobolev's inequality, by
$$
\int_{\R^3} \left|\nabla |\psi|_\varepsilon^{1/2}\right|^2\,dx 
= \int_{\R^3} \left|\nabla \left(|\psi|_\varepsilon^{1/2} -\varepsilon^{1/2}\right) \right|^2\,dx 
\geq S_3 \left( \int_{\R^3} \left( |\psi|_\varepsilon^{1/2}-\varepsilon^{1/2} \right)^6 dx \right)^{1/3} \,,
$$
and  the first term on the right side we bound from above using
$$
\left\langle \psi,\sigma\cdot B\psi \right\rangle \leq |B| \left|\left\langle\psi,\sigma\psi\right\rangle\right| = |B| |\psi|^2 \,.
$$
Thus, we have
	$$
	S_3 \left( \int_{\R^3} \left( |\psi|_\varepsilon^{1/2}-\varepsilon^{1/2} \right)^6 dx \right)^{1/3} \leq \frac12 \int_{\R^3} |B| \frac{|\psi|^2}{|\psi|_\varepsilon}\,dx \,.
	$$
	On the right side, we can bound $|\psi|_\varepsilon^{-1}\leq |\psi|^{-1}$. On the left side, since $\varepsilon\mapsto (m+\varepsilon^2)^{1/4}- \varepsilon^{-1/2}$ is decreasing, we can apply the monotone convergence
	theorem to obtain
	$$
	S_3 \left( \int_{\R^3} |\psi|^3 dx \right)^{1/3} \leq \frac12 \int_{\R^3} |B| |\psi|\,dx \,,
	$$
	which, by means of H\"older's inequality, implies the assertion.
\end{proof}

\section{Is Theorem \ref{magnetic} optimal?} \label{optimality}

For the zero modes given by \eqref{lossyau1}-\eqref{lossyau3} we find $\Vert B \Vert_{3/2} = 4\,S_3$ which raises the question whether or not $4\,S_3$ is the sharp constant in Theorem \ref{magnetic}. While we do not know whether Theorem~\ref{magnetic} is sharp, the following theorem elucidates this point.

Recall that a measurable function $f$ on $\R^3$ is in the weak $L^p$ space, $L^p_w(\R^3)$, if
$$
\Vert f \Vert_{w,p} := \sup_{t>0} |\{ |f| > t\}|^{1/p} t < \infty \ .
$$
It is well known that this expression is not a norm, but is equivalent to one for $p>1$. Also, in terms of the symmetric decreasing rearrangement $|f|^*$ of $|f|$, one has
$$
\Vert f \Vert_{w,p} = \left( \frac{4\pi}{3} \right)^{1/p} \sup_{x \in \R^3} |x|^{3/p} |f|^*(x) \,.
$$
 To see this identity, note that for any $t>0$ there is an $R_t>0$ such that $\{|x|<R_t\}=\{|f|^*>t\}$ and therefore $(4\pi/3)R_t^3 = |\{ |f| >t\}|$. We have, at least at points of continuity of $|f|^*$, $|f|^*(x)=t$ if $|x|=R_t$, so
$$
\sup_{x\in\R^3} |x|^{3/p} |f|^*(x) = \sup_{t>0} R_t^{3/p} t = (4\pi/3)^{-1/p} \sup_t  |\{|f|>t\}|^{1/p} t = (4\pi/3)^{-1/p} \Vert f \Vert_{w,p} \,.
$$

\begin{theorem} \label{magneticweak}
	Let $B \in L^{3/2}_w(\R^3: \R^3)$ be a magnetic field, i.e., ${\rm div}\, B=0$. If \eqref{zeromode} has a weak solution $0\not\equiv \psi \in H^1(\R^3: \C^2)$, then
$$
\Vert B \Vert_{w, 3/2}  \ge \frac12 \left( \frac{4\pi}{3}\right)^{2/3}
$$
or, equivalently,
\begin{equation}\label{weakbound}
\sup_{x\in\R^3} |x|^2 |B|^*(x) \ge \frac 12 \ .
\end{equation}
\end{theorem}

We shall argue  that this result is in some sense sharp by showing that `zero modes' of the magnetic monopole saturate the inequality \eqref{weakbound}.

We start with a few simple observations regarding the ground state problem for the Schr\"odin\-ger equation.
Recall that Sobolev's inequality shows that, for 
the Schr\"odin\-ger operator $-\Delta -V, V\ge 0$,  to have a bound state, i.e., a negative energy solution that is in $L^2(\R^3)$, a necessary condition is that $\Vert V \Vert_{3/2} \ge S_3$. This result is sharp. Another scale invariant quantity that is indicative for the existence of a bound state is the weak norm $\Vert V \Vert_{w,3/2}$. Indeed, using rearrangements \cite[Theorems 3.4 and 7.17]{LL} and Hardy's inequality one arrives at 
\begin{align} \label{hardy}
	\frac14 \int_{\R^3}\frac{ (f^*)^2 }{|x|^2} \,dx & \leq \int_{\R^3} |\nabla f^*|^2 \,dx \leq 
\int_{\R^3} |\nabla f|^2\,dx \le \int_{\R^3} V f^2\,dx \notag \\
& \leq \int_{\R^3} V^* (f^*)^2 \,dx \le 
\left( \sup_{x\in\R^3} |x|^2 V^*(x)  \right) \int_{\R^3}\frac{ (f^*)^2 }{|x|^2} \,dx \,.
\end{align}
Hence
\begin{equation} \label{optimalsob}
	\left( \frac{4\pi}3 \right)^{-2/3} \| V \|_{w,3/2} = \sup_{x\in\R^3} |x|^2 V^*(x) \geq \frac14
\end{equation}
is a necessary condition for the existence of a bound state.

Another point one should make is that the function $|x|^{-1/2}$ is a solution of the Sobolev equation
\begin{equation*}\label{sobequation}
-\Delta f = \frac 14 f^5
\end{equation*}
and this solution is in $L^6_w(\R^3)$. Thus, one may ask for a necessary condition on $\Vert V \Vert_{w,3/2}$ so that the Schr\"odinger operator has a bound state in $L^6_w(\R^3)$. A bound state in this context is a sub-solution $0\le  f \in L^6_w(\R^3)$, i.e., $-\Delta f -Vf \le 0$, and such that $\Delta f \in L^1_{loc}(\R^3)$. We claim that under these weaker assumptions, we still have \eqref{optimalsob}. To see this, once more, multiplying the sub-solution inequality by $f_c = \min\{ \frac1c, [f-c]_+ \}$, where $[\cdot]_+$ denotes the positive part and $c>0$ is small, we obtain
$$
\int_{\R^3} |\nabla f_c|^2 \,dx = - \int_{\R^3} f_c\Delta f \,dx \le \int_{\R^3} V f_c f\, dx \ .
$$
Here, the integration by parts is justified since $f_c\Delta f \in L^1(\R^3)$ \cite[Theorem 7.7]{LL}. Now arguing similarly as before, using rearrangment inequalities and Hardy's inequality and noting that $(f^*)_c = (f_c)^* =: f_c^*$,
\begin{align}\label{eq:hardy2}
	\frac14 \int_{\R^3}\frac{ (f^*_c)^2 }{|x|^2} \,dx & \leq \int_{\R^3} |\nabla f^*_c|^2 \,dx \leq \int_{\R^3} |\nabla f_c|^2\,dx \le \int_{\R^3} V f_c f\,dx \notag \\
	& \leq \int_{\R^3} V^* f_c^* f^* \,dx \le 
	\left( \sup_{x\in\R^3} |x|^2 V^*(x)  \right) \int_{\R^3}\frac{ f_c^* f^*}{|x|^2} \,dx \,.
\end{align}
If $\int_{\R^3} \frac{(f^*)^2}{|x|^2}  dx < \infty$, then monotone convergence yields \eqref{hardy} and hence \eqref{optimalsob}. If $\int_{\R^3} \frac{(f^*)^2}{|x|^2}  dx = \infty$, we note that
$$
\int_{\R^3} \frac{f_c^*f^*}{|x|^2}dx = c \int_{c \le f^* \le c+1/c} \frac{f_c^*}{|x|^2}dx + \int_{c\le f^* \le c+ 1/c} \frac{(f_c^*)^2}{|x|^2}dx +\frac1c  \int_{f^*> c+ 1/c} \frac{f^*}{|x|^2}dx \ .
$$
Since $f^*$ is symmetric decreasing and belongs to $L^6_w(\R^3)$, we have that
$$
f^*(x) \le \frac{D}{|x|^{1/2}}
$$
for some constant $D$. Simple estimates then show that
$$
c \int_{c \le f^* \le c+ 1/c} \frac{f_c^*}{|x|^2}dx \le D^2 \qquad {\rm and} \qquad
\frac1c  \int_{f^*>c+  1/c} \frac{f^*}{|x|^2}dx \le 2D^2 \ .
$$
From these estimates we glean that
$$
\lim_{c \to 0} \frac{\int_{\R^3} \frac{f_c^*f^*}{|x|^2}dx}{\int_{\R^3} \frac{(f_c^*)^2}{|x|^2}dx} = 1
$$
which, when inserted into \eqref{eq:hardy2}, yields again \eqref{optimalsob}. This shows our claim that \eqref{optimalsob} holds under the weaker sub-solution assumptions.

Using the singular solution $|x|^{-1/2}$ of the Sobolev equation \eqref{sobequation} we see that \eqref{optimalsob} is sharp!

\begin{proof}[Proof of Theorem \ref{magneticweak}]
Because $\psi \in H^1(\R^3: \C^2)$ we can trace the steps leading to inequality \eqref{crux}. Instead of using the Sobolev inequality we use Hardy's inequality and argue as above.
\end{proof}

\begin{remark}
	The assumption $\psi\in H^1(\R^3:\C^2)$ in Theorem \ref{magneticweak} can be replaced by the formally weaker assumption that $|\{|\psi|>\tau \}|<\infty$ for every $\tau<\infty$ and
	\begin{equation}
		\label{eq:lorentzass}
		\int_{\R^3} \frac{(|\psi|^*)^2}{|x|^2}\,dx <\infty \,.
	\end{equation}
	It is well-known and easy to see that assumption \eqref{eq:lorentzass} is equivalent to $\psi$ belonging to the Lorentz space $L^{6,2}(\R^3)$. It is a simple consequence of Hardy's inequality that $\psi\in H^1(\R^3:\C^2)$ implies \eqref{eq:lorentzass}, so \eqref{eq:lorentzass} is formally a weaker assumption. On the other hand, if for $B$ as in the theorem we choose $A\in L^3_w(\R^3,\R^3)$ with ${\rm curl}A =B$, then
	$$
	\int_{\R^3} |\sigma\cdot A \psi |^2\,dx \leq \int_{\R^3} |A|^2|\psi|^2\,dx \leq \int_{\R^3} (|A|^*)^2 (|\psi|^*)^2\,dx \leq (\sup_{\R^3} |x| |A|^*)^2 \int_{\R^3} \frac{(|\psi|^*)^2}{|x|^2}\,dx \,,
	$$
	and so, if $\psi$ is a solution of \eqref{zeromode} satisfying \eqref{eq:lorentzass}, then $-i\sigma\nabla\psi = \sigma\cdot A \psi\in L^2(\R^3)$. As discussed before Lemma \ref{diamagref}, this implies $\nabla\psi\in L^2(\R^3)$, which is enough for the proof of Theorem \ref{magneticweak}.	
\end{remark}

It turns out that there is a solution for the zero mode equation that is analogous to \eqref{sobequation}.
We shall consider  the magnetic field of a monopole, which is strictly speaking not a standard magnetic field
but serves to explain some of the structure of our problem. 
Starting with the spinor
$$
\psi = \frac{1}{\sqrt 2 r^{3/2}} \left( \begin{array}{c} \sqrt{r+z} \\ \frac{x+iy}{\sqrt{r+z}}  \end{array}\right) \ ,
$$
where $r = \sqrt{x^2+y^2+z^2}$,  we easily verify
$$
\sigma\cdot \vec x \psi = r \psi
$$
 and
$$
|\psi|^2 = \frac{1}{r^2} \ .
$$ 
Now consider the monopole $A$-field
$$
A = g\frac{(-y,x,0)}{r(r+z)}
$$
with a parameter $g$ representing the monopole strength. As always, one has to exclude the negative $z$-axis, $\{ z\leq 0\}$. There is an analogous formula for the vector potential $A'$
where the positive $z$-axis, $\{ z\geq 0\}$, has to be excluded. The fields $A$ and $A'$ differ by a gauge in the complement of 
the $z$-axis.
Again a simple computation yields
$$
{\rm curl} A = g\frac{\vec x}{r^3} \ .
$$
We also have that
$$
\sigma \cdot A \psi = i g \frac{1}{\sqrt 2 r^{5/2}} \left(\begin{array}{c} -\frac{r-z}{\sqrt{r+z}} \\ \frac{x+iy}{\sqrt{r+z}} \end{array}\right)
$$
and
$$
\sigma \cdot (-i\nabla) \psi = i\frac{1}{2}\frac{1}{\sqrt 2 r^{5/2}}\left( \begin{array}{c} -\frac{r-z}{\sqrt{r+z}} \\ \frac{x+iy}{\sqrt{r+z}} \end{array}\right),
$$
from which we get that
$$
\sigma \cdot \left(-i \nabla -A\right) \psi = 0
$$
{\it if we choose} $g=\frac{1}{2}$. This is the smallest value for the monopole strength.
Since
$$
|\psi|^{1/2} = r^{-1/2} \ ,
$$
and
$$
|\nabla |\psi|^{1/2}|^2 = \frac14 \frac{1}{r^3}
$$
we find
$$
 \frac{1}{2} \frac{1}{|\psi|}\langle \psi, \sigma\cdot B \psi \rangle= \frac{1}{2} g \frac{1}{r^2} |\psi|= \frac{1}{4} \frac{1}{r^3} \ ,
$$
since $g=1/2$. Hence
$$
|\nabla |\psi|^{1/2}|^2 = \frac12 B \cdot \langle \frac{\psi}{|\psi|}, \sigma \psi\rangle 
$$
which is a pointwise inequality for the integrants in \eqref{crux}. Note that the magnetic monopole field $x |x|^{-3}$ is in $L^{3/2}_w(\R^3:\R^3)$. Likewise, the spinor  is in $L^3_w(\R^3: \C^2)$ (but it is not in $H^1(\R^3:\C^2)$ and does not satisfy \eqref{eq:lorentzass}). This situation is very analogous to the scalar case.  The monopole field with strength $1/2$ clearly satisfies the condition \eqref{weakbound}, in fact with equality, and hence, if we allow monopole fields into our considerations, we learn that \eqref{weakbound} is sharp for the existence of zero modes.


\section{Proof of Theorem \ref{spinor}}

We assume that $-i\sigma\cdot\nabla\psi = 3\lambda\psi$ for a spinor $\psi\in L^p(\R^3:\C^2)$ for some $3/2<p<\infty$ and a real function $\lambda\in L^3(\R^3)$. By a straightforward modification of the proof of Theorem \ref{regularity}, we have $\psi \in L^r(\R^3:\C^2)$ for all $3/2 <  r < \infty$. Therefore, by the same argument as at the beginning of the proof of Theorem \ref{magnetic},  $\psi \in H^1(\R^3:\C^2)$.

We consider again the operator
$$
\Pi(\alpha \otimes \psi)_j = \alpha_j \psi - \frac{1}{3} \sigma_j \sigma \cdot \alpha \psi \,,
$$
but proceed in a slightly different manner by considering
$$
\Pi(\partial_j - i\lambda(x) \sigma_j)\phi= (\partial_j - i\lambda(x) \sigma_j)\phi - \frac{1}{3} \sigma_j  \sigma\cdot ( \nabla -i \lambda(x) \sigma) \phi \ .
$$
\begin{lemma}\label{diamagref2}
Let $\psi \in L^p(\R^3:\C^2)$, $3/2<p <\infty$ satisfy $-i \sigma \cdot \nabla \psi =3 \lambda(x) \psi$. Then $\psi \in H^1(\R^3:\C^2)$,
and $|\psi| \in H^1(\R^3)$ as well and moreover,  almost everywhere in $\R^3$,
\begin{equation} \label{modified}
\left|\nabla |\psi|\right|^2 \leq \frac 23\ \left|[\nabla-i\lambda(x) \sigma ]\psi \right|^2 \,.
\end{equation}
\end{lemma}
\begin{proof}
Since the proof is technically the same as the proof of Lemma \ref{diamagref} we just indicate the changes.
The starting point is the identity
$$
|\nabla |\psi|| =  {\rm Re} \langle \frac{\nabla |\psi| \psi}{|\nabla |\psi|| |\psi|}, \nabla \psi \rangle
=  {\rm Re} \langle \frac{\nabla |\psi| \psi}{|\nabla |\psi|| |\psi|}, [\nabla - i\lambda(x) \sigma] \psi \rangle
=  {\rm Re} \langle \frac{\nabla |\psi| \psi}{|\nabla |\psi|| |\psi|}, \Pi [\nabla - i\lambda(x) \sigma] \psi \rangle \,,
$$
since $\psi$ solves the spinor equation. The claimed inequality then follows as in the proof of Lemma \ref{diamagref}. Again, these steps can be made rigorous by considering the function
$|\psi|_\varepsilon$ and then taking the limit $\varepsilon \to 0$. 
\end{proof} 
We continue with the proof of Theorem \ref{spinor}.  By Lemma \ref{squareroot},
$$
\left|\nabla |\psi|_\varepsilon^{1/2}\right|^2 = \frac12 \left( \re \left\langle \nabla \frac{\psi}{|\psi|_\varepsilon},\nabla\psi \right\rangle - |\psi|_\varepsilon^{-1} |\nabla\psi|^2 \right) + \frac{3}{4} \frac{|\psi|^2}{|\psi|_\varepsilon^3} \left|\nabla|\psi|\right|^2 
$$
$$
=\frac12 \left( \re \left\langle [\nabla - i\lambda \sigma]  \frac{\psi}{|\psi|_\varepsilon}, [\nabla- i\lambda \sigma]  \psi \right\rangle - |\psi|_\varepsilon^{-1} |[\nabla - i\lambda \sigma] \psi|^2 \right) + \frac{3}{4} \frac{|\psi|^2}{|\psi|_\varepsilon^3} \left|\nabla|\psi|\right|^2 \,.
$$
noting that 
 \begin{equation} \label{spin}
 {\rm Re} \left(\nabla \left( \frac{1}{|\psi|_\varepsilon}\right) \cdot \langle \psi, i\sigma \psi\rangle\right) = 0 \ .
 \end{equation}
 Lemma \ref{diamagref2} then yields the estimate
 $$
 \left|\nabla |\psi|_\varepsilon^{1/2}\right|^2 \le \frac12 \left( \re \left\langle [\nabla - i\lambda \sigma]  \frac{\psi}{|\psi|_\varepsilon}, [\nabla- i\lambda \sigma]  \psi \right\rangle - \frac32 |\psi|_\varepsilon^{-1} |\nabla |\psi||^2 \right) + \frac{3}{4} \frac{|\psi|^2}{|\psi|_\varepsilon^3} \left|\nabla|\psi|\right|^2
 $$
which simplifies to
$$
 \left|\nabla |\psi|_\varepsilon^{1/2}\right|^2 \le \frac12 \re \left\langle [\nabla - i\lambda \sigma]  \frac{\psi}{|\psi|_\varepsilon}, [\nabla- i\lambda \sigma]  \psi \right\rangle  - \frac{3}{4} \frac{\varepsilon^2}{|\psi|_\varepsilon^3} \left|\nabla|\psi|\right|^2 \ .
$$
Integrating this inequality, taking into account that $\frac{\psi}{|\psi|_\varepsilon}  \in H^1(\R^3:\C^2)$ by Remark \ref{psiepsrem}, we find
$$
\int_{\R^3} \left|\nabla |\psi|_\varepsilon^{1/2}\right|^2 dx \le \frac12 \int_{\R^3}  \re \left\langle [\nabla - i\lambda \sigma]  \frac{\psi}{|\psi|_\varepsilon}, [\nabla- i\lambda \sigma]  \psi \right\rangle dx
$$
which, on account of \eqref{spin}, Lemma \ref{byparts} (with $A=0$) and the spinor equation reduces to
$$
\int_{\R^3} \left|\nabla |\psi|_\varepsilon^{1/2}\right|^2 dx \le 3 \int_{\R^3} \lambda^2 \frac{|\psi|^2}{|\psi|_\varepsilon} dx \ .
$$
Continuing as in the proof of Theorem \ref{magnetic} we find, using Sobolev's and H\"older's inequality, that
$$
S_3 \le 3  \left( \int_{\R^3} | \lambda|^3 dx \right)^{2/3} \ .
$$
This is the claimed inequality. Finally, we note that we have equality for $\lambda(x) = \frac{1}{1+x^2}$, since
$$
\left( \int_{\R^3} \left( \frac{1}{1+x^2}\right)^3 dx \right)^{2/3} = \frac{1}{4} \left( \int_{\R^3} \left( \frac{2}{1+x^2}\right)^3 dx \right)^{2/3} = \frac{1}{4}\, |\Sp^3|^{2/3} = \frac13\, S_3 \,.
$$


\section{Sketch of a proof of Theorems \ref{genmagnetic} and   \ref{spinorgeneral}}
\label{sec:higherdim}

First the proof of Theorem \ref{spinorgeneral}: With the same argument as in the proof of Theorem \ref{regularity} one can show that $\psi \in L^r$ for all $\frac{d}{d-1} < r < \infty$. In particular, $\psi \in L^{2d / (d-2)}$ and therefore $\psi \in H^1(\R^d : \C^{2^\nu})$. Recall that $\nu=(d-1)/2$ if $d$ is odd and $\nu=d/2$ if $d$ is even. Since the steps in the proof of Theorem \ref{spinorgeneral} completely analogous to the proof of Theorem \ref{spinor} 
we just give a sketch of the argument. Recall that $d$ is the dimension of the underlying space. The projection $\Pi(\alpha \otimes \psi)$ is now given by
$$
\Pi(\alpha \otimes \psi) = \alpha_i \psi - \frac{1}{d} \gamma_i (\alpha \cdot \gamma) \psi
$$
from which one easily gleans the estimate
$$
|\Pi(\alpha \otimes \psi)|^2 \le \frac{d-1}{d} |\alpha|^2|\psi|^2 \ .
$$
Likewise, the equation
$$
\Pi(\partial_j - i\lambda(x) \gamma_j)\phi= (\partial_j - i\lambda(x) \gamma_j)\phi - \frac{1}{d} \gamma_j  \gamma \cdot ( \nabla -i \lambda(x) \gamma) \phi \ .
$$
reduces to 
$$
\Pi(\partial_j - i\lambda(x) \gamma_j)\phi= (\partial_j - i\lambda(x) \gamma_j)\phi \ ,
$$
if $\phi$ satisfies the spinor equation \eqref{generalequation}.
Analogous to Lemma \ref{diamagref2} we have
\begin{equation} \label{diamaggen}
d |\nabla |\psi||^2 \le (d-1) |(\nabla -i\lambda(x) \gamma )\psi|^2 \ .
\end{equation}
By a straightforward computation in line with  Lemma \ref{squareroot} one finds
\begin{equation} \label{rootgeneral}
{\rm Re} \langle \nabla \frac{\psi}{|\psi|_\varepsilon^{\frac{2}{d-1}}}, \nabla \psi\rangle = \frac{|\nabla \psi|^2}{|\psi|_\varepsilon ^{\frac{2}{d-1}}} - \frac{2}{d-1} \frac{|\psi|^2}{|\psi|_\varepsilon^{2+\frac{2}{d-1}}} |\nabla |\psi||^2
\end{equation}
and a further computation yields
$$
{\rm Re} \langle [\nabla - i\lambda \gamma ] \frac{\psi}{|\psi|_\varepsilon^{\frac{2}{d-1}}}, [\nabla -i\lambda \gamma]\psi\rangle = \frac{|[\nabla -i\lambda \gamma]  \psi|^2}{|\psi|_\varepsilon ^{\frac{2}{d-1}}} - \frac{2}{d-1} \frac{|\psi|^2}{|\psi|_\varepsilon^{2+\frac{2}{d-1}}} |\nabla |\psi||^2 \,.
$$
Using \eqref{diamaggen}, this leads to the inequality
$$
{\rm Re} \langle [\nabla - i\lambda \gamma ] \frac{\psi}{|\psi|_\varepsilon^{\frac{2}{d-1}}}, [\nabla -i\lambda \gamma]\psi\rangle \ge \frac{d}{d-1} \frac{ |\nabla |\psi||^2}{|\psi|_\varepsilon ^{\frac{2}{d-1}}} - \frac{2}{d-1} \frac{|\psi|^2}{|\psi|_\varepsilon^{2+\frac{2}{d-1}}} |\nabla |\psi||^2 \,,
$$
which, in turn, is bounded below by
$$
 \frac{d-2}{d-1} \frac{|\psi|^2}{|\psi|_\varepsilon^{2+\frac{2}{d-1}}} |\nabla |\psi||^2 \,.
$$
Set $q :=1- \frac{1}{d-1} =\frac{d-2}{d-1}$ and compute, using the chain rule for Sobolev functions,
$$
|\nabla |\psi|^q_\varepsilon|^2 = q^2 \frac{|\psi |^2}{|\psi|_\varepsilon^{4-2q} }|\nabla |\psi||^2 = \left(\frac{d-2}{d-1}\right)^2 \frac{|\psi |^2}{|\psi|_\varepsilon^{2+\frac{2}{d-1}} }|\nabla |\psi||^2 \,,
$$
which yields
$$
{\rm Re} \langle [\nabla - i\lambda \gamma ] \frac{\psi}{|\psi|_\varepsilon^{\frac{2}{d-1}}}, [\nabla -i\lambda \gamma]\psi\rangle \ge \frac{d-1}{d-2} |\nabla |\psi|_\varepsilon^q|^2 \ .
$$
Integrating this inequality and using the analog of Lemma \ref{byparts} yields
$$
\frac{d-1}{d-2} \int_{\R^d} |\nabla |\psi|_\varepsilon^q|^2 \,dx 
 \le \int_{\R^d}  {\rm Re} \langle [\nabla - i\lambda \gamma ] \frac{\psi}{|\psi|_\varepsilon^{\frac{2}{d-1}}}, [\nabla -i\lambda \gamma]\psi\rangle dx 
 $$
 and, by proceeding with reasoning similar to the one in the previous section, the right side equals
 $$
 d(d-1) \int_{\R^d} |\psi|^{2q}  \lambda^2 dx \,.
 $$
 Using H\"older's inequality, we therefore obtain
$$
\int_{\R^d} |\nabla |\psi|_\varepsilon^q|^2 \,dx  \le d(d-2) \left( \int_{\R^d} |\psi|^{q \frac{2d}{d-2}} \,dx \right)^{\frac{d-2}{d}} \left( \int_{\R^d} |\lambda(x)|^d \,dx\right)^{\frac{2}{d}}  \ .
$$
We bound the left side from below by Sobolev's inequality
$$
 \int_{\R^d} |\nabla |\psi|_\varepsilon^q|^2 dx  \ge S_d  \left( \int_{\R^d} |\psi|^{q \frac{2d}{d-2}} dx \right)^{\frac{d-2}{d}} \,.
 $$
 If $\psi\not\equiv 0$, we conclude that
 $$
  \left( \int_{\R^d} |\lambda(x)|^d dx\right)^{\frac{2}{d}}  \ge \frac1{d(d-2)}\, S_d \,,
$$
 which is the statement of the theorem. The Dunne-Min spinor (and its obvious generalization to even dimensions) satisfies the equation in the theorem with $\lambda(x) = \frac{1}{1+|x|^2}$ and it is easy to check that this yields equality in the above inequality.

The proof of Theorem \ref{genmagnetic}  is analogous. Using \eqref{rootgeneral}, a simple computation yields
$$
{\rm Re} \langle (\nabla -i A) \frac{\psi}{|\psi|^{\frac{2}{d-1}}_\varepsilon}, (\nabla - iA)\psi\rangle =
\frac{|(\nabla -iA)\psi|^2}{|\psi|^{\frac{2}{d-1}}_\varepsilon} -\frac{2}{d-1} \frac{|\psi|^2}{|\psi|^{2+\frac{2}{d-1}}_\varepsilon}
|\nabla|\psi||^2 \,,
$$
which, when combined with the analogue of the inequality in Lemma \ref{diamagref} to higher dimensions, namely
$$
|\nabla |\psi||^2 \le \frac{d-1}{d} |(\nabla -iA)\psi|^2 \,,
$$
yields
$$
{\rm Re} \langle (\nabla -i A) \frac{\psi}{|\psi|^{\frac{2}{d-1}}_\varepsilon}, (\nabla - iA)\psi\rangle \ge \frac{d-2}{d-1} \frac{|\nabla |\psi||^2}{|\psi|^{\frac{2}{d-1}}_\varepsilon} \,.
$$
The analogue of Lemma \ref{byparts} reads
$$
\int_{\R^d} \frac{1}{|\psi|^{\frac{2}{d-1}}_\varepsilon} {\rm Re}  \left\{\sum_{j < k} [\partial_jA_k-\partial_k A_j]  \langle \psi, i \gamma_j \gamma_k \psi \rangle \right\} dx + \int_{\R^d} {\rm Re}  \langle \gamma \cdot (\nabla -iA)\frac{\psi}{|\psi|^{\frac{2}{d-1}}_\varepsilon}, \gamma \cdot (\nabla -iA) \psi\rangle dx 
$$
$$
= \int_{\R^d} {\rm Re} \langle (\nabla -i A) \frac{\psi}{|\psi|^{\frac{2}{d-1}}_\varepsilon}, (\nabla - iA)\psi\rangle \,dx \,.
$$ 
Hence, by the zero mode equation,
$$
\int_{\R^d} \frac{1}{|\psi|^{\frac{2}{d-1}}_\varepsilon} {\rm Re}  \left\{\sum_{j < k} [\partial_jA_k-\partial_k A_j]  \langle \psi, i \gamma_j \gamma_k \psi \rangle \right\} dx \ge \frac{d-2}{d-1} \int_{\R^d}  \frac{|\nabla |\psi||^2}{|\psi|^{\frac{2}{d-1}}_\varepsilon} \,dx
$$
It is shown in Appendix \ref{genBfield} that
$$
{\rm Re}  \left\{\sum_{j < k} [\partial_jA_k-\partial_k A_j]  \langle \psi, i \gamma_j \gamma_k \psi \rangle \right\}
\le \nu^{1/2}  |\psi|^2 |B| \,,
$$
where we recall the definition of $|B|$ given in the statement of the theorem and the notation $\nu=(d-1)/2$ if $d$ is odd and $\nu=d/2$ if $d$ is even. Thus, we get the estimate
$$
\frac{d-2}{d-1} \int_{\R^d}  \frac{|\nabla |\psi||^2}{|\psi|^{\frac{2}{d-1}}_\varepsilon}\,dx \le  \nu^{1/2}  \int_{\R^d}\frac{|\psi|^2}{|\psi|^{\frac{2}{d-1}}_\varepsilon} |B| \,dx \ .
$$ 
Simple computations using Sobolev's inequality then yield the result.
 
\section{Proof of Theorem \ref{improvedz}}
The proof in \cite{FLL} is based on two ingredients, namely the diamagnetic inequality and a certain Hardy--Sobolev inequality. Here we modify both these inputs, namely, we use the improved diamagnetic inequality for zero modes and we use a different Hardy--Sobolev inequality for which we can determine the sharp constant.

For comparison we review the argument in \cite{FLL}. If $\sigma\cdot(-i\nabla-A)\psi=0$, then, with $B:={\rm curl} A$,
$$
0 = \int_{\R^3} |\sigma\cdot(-i\nabla-A)\psi|^2 \,dx = \int_{\R^3} |(-i\nabla -A)\psi|^2\,dx - \int_{\R^3} B\cdot\langle\psi,\sigma\psi\rangle\,dx \,,
$$
that is,
$$
\int_{\R^3} |(-i\nabla -A)\psi|^2\,dx = \int_{\R^3} B\cdot\langle\psi,\sigma\psi\rangle\,dx \,.
$$
By the diamagnetic inequality, the left side is bounded from below by $\int_{\R^3} |\nabla |\psi||^2\,dx$, while, in view of $B\cdot\langle\psi,\sigma\psi\rangle\leq |B| |\psi|^2$, the right side is bounded from above by $\|B\|_2 \|\psi\|_4^2$. Thus, setting $u=|\psi|$ and dropping the constraint that $\psi$ is a zero mode, we obtain
$$
\hat z := 8\pi\alpha^2 Z_c \geq \inf\left\{ \frac{\|\nabla u\|_2^4\,}{\|u\|_4^4\, (u,|x|^{-1} u)} :\ u\in H^1(\R^3)\,,\ \|u\|_2=1 \right\} \ .
$$
The right side can be thought of as the sharp constant in a certain Hardy--Sobolev inequality. Fr\"ohlich, Lieb and Loss do not compute this constant explicitly, but they bound it using the hydrogen uncertainty principle $\|\nabla u\|_2 \|u\|_2 \geq (u,|x|^{-1} u)$ and the Sobolev interpolation inequality $\|\nabla u\|_2^{3/2} \|u\|_2^{1/2} \geq S \|u\|_4^2$ with a numerical value for the constant $S$. The authors also observe that by combining these two sharp inequalities they obtain a constant which is very close to the sharp constant in the more complicated Hardy--Sobolev inequality.

We now turn to the proof of our improved bound. It consists essentially in showing that
$$
\hat z \geq 4\, \inf\left\{ \frac{\|\nabla u\|_2^4}{(|u|^2,|x|^{-1}|u|^2)} :\ u\in \dot H^1(\R^3) \right\}
$$
and computing the infimum on the right side explicitly.

\begin{proof}[Proof of Theorem \ref{improvedz}]
As we have shown in the proof of Theorem \ref{magnetic}, if $\psi$ is a normalized zero mode, then
$$
\int_{\R^3} |\nabla |\psi|_\epsilon^{1/2}|^2\,dx \leq \frac12 \int_{\R^3} |B| \frac{|\psi|^2}{|\psi|_\epsilon}\,dx \,,
$$
where $|\psi|_\epsilon = \sqrt{|\psi|^2+\epsilon^2}$. We bound the right side from above by
$$
\frac12 \int_{\R^3} |B| \frac{|\psi|^2}{|\psi|_\epsilon}\,dx \leq \frac12 \|B\|_2 \|\psi\|_2 = \frac12 \|B\|_2 \,.
$$
On the other hand, using the Sobolev inequality in Theorem \ref{hardysobolev} in the next section, we can bound the left side from below by
$$
\int_{\R^3} |\nabla |\psi|_\epsilon^{1/2}|^2\,dx  \geq \sqrt{\frac{8\pi}{3}} \left( \int_{\R^3} \frac{(|\psi|_\epsilon^{1/2}-\epsilon^{1/2})^4}{|x|}\,dx \right)^{1/2} \,.
$$
Thus, we obtain
$$
\frac{\int_{\R^3} |B(x)|^2\,dx}{\int_{\R^3} \frac{(|\psi|_\epsilon^{1/2}-\epsilon^{1/2})^4}{|x|}\,dx} \geq \frac{32\,\pi}{3} \,.
$$
By dominated convergence, this gives
$$
\frac{\int_{\R^3} |B(x)|^2\,dx}{\int_{\R^3} \frac{|\psi|^2}{|x|}\,dx} \geq \frac{32\,\pi}{3} \,,
$$
and, recalling \eqref{criticalzee}, implies the theorem.
\end{proof}


\section{A sharp Hardy--Sobolev inequality}

In the previous section we used the following sharp inequality.

\begin{theorem}\label{hardysobolev}
For any $u\in\dot H^1(\R^3)$,
$$
\int_{\R^3} |\nabla u|^2 \,dx \geq \sqrt{\frac{8\pi}{3}} \left( \int_{\R^3} \frac{|u|^4}{|x|}\,dx \right)^{1/2}.
$$
Equality holds if and only if $u$ is a multiple or dilate of
$$
(1+|x|)^{-1} \,.
$$
\end{theorem}

We deduce this theorem from the following well-known one-dimensional inequality. The sharp constant was computed, for instance, by Sz.-Nagy in 1941, \cite{Sz.-N}.

\begin{lemma}\label{sobolevnagy}
For any $f\in H^1(\R)$,
$$
\int_\R \left( |f'|^2 + \frac14 |f|^2\right)dt \geq \sqrt{\frac{2}{3}} \left( \int_\R |f|^4\,dt \right)^{1/2}.
$$
Equality holds if and only if $f$ is a multiple or translate of
$$
(\cosh(t/2))^{-1} \,.
$$
\end{lemma}

\begin{proof}[Proof of Theorem \ref{hardysobolev}]
By rearrangement \cite[Theorems 3.4 and 7.17]{LL} it suffices to prove the inequality for radial functions $u$. For the latter, the claimed inequality becomes
$$
\int_0^\infty |\partial_r u|^2 r^2\,dr \geq \sqrt{\frac{2}{3}} \left( \int_0^\infty |u|^4 r^2\,dr \right)^{1/2}.
$$
Note that, since the weight $|x|^{-1}$ is strictly decreasing, the rearrangement inequality is strict \cite[Theorem 3.4]{LL} and therefore equality in the three-dimensional inequality holds if and only if equality holds in the one-dimensional inequality for a radial, non-increasing function.

Finally, we set $u(r) = r^{-1/2} f(\ln r)$ and note that $\partial_r u = r^{-3/2}(f'(\ln r) - (1/2) f(\ln r))$. Thus, by a change of variables,
$$
\int_0^\infty |\partial_r u|^2 r^2\,dr = \int_\R \left( |f'|^2 + \frac14 |f|^2\right)dt
\qquad\text{and}\qquad
\int_0^\infty |u|^4 r^2\,dr = \int_\R |f|^4\,dt \,.
$$
Therefore, the theorem is a consequence of Lemma \ref{sobolevnagy}.
\end{proof}


\section{Some open problems}

Another take on the problem whether a magnetic field can support a zero mode is the following. 
Starting from a zero mode one finds, using H\"older's inequality, that
$$
\Vert -i \sigma\cdot \nabla  \psi\Vert_{3/2} \le \Vert A \Vert_3 \Vert \psi \Vert_3 \,.
$$
Using the Hardy-Littlewood-Sobolev inequality it is easy to see that there exists a constant $C_s>0$ such that
\begin{equation} \label{spinorinequality}
\Vert -i \sigma \cdot \nabla  \psi \Vert_{3/2} \ge C_s \Vert\psi \Vert_3 \ .
\end{equation}
Likewise, it is not hard to see that there must be a constant $C_f>0$ such that
\begin{equation} \label{fieldinequality}
\Vert {\rm curl} A \Vert_{3/2} \ge C_f \inf_\phi \Vert A +\nabla \phi \Vert_3  \ .
\end{equation}
The functional $\phi \mapsto \Vert A +\nabla \phi \Vert_3$ is convex and hence there is a minimizer, $\phi_0$, and the Euler-Lagrange equation is
$$
{\rm div} |A+\nabla \phi_0| (A+\nabla \phi_0) = 0 \ .
$$
Thus we have that
\begin{equation} \label{field}
\Vert {\rm curl} A \Vert_{3/2} \ge C_f \Vert A \Vert_3
\end{equation}
where we impose the additional constraint
$$
{\rm div} (|A|A) = 0 \ .
$$
From this, we readily see that a necessary condition for the existence of a zero mode is that
$$
\Vert B \Vert_{3/2} \ge C_sC_f \,.
$$
Thus, it remains to determine the sharp constants in the inequalities. We do not know how to do this but give some results that point to  interesting connections with other areas of mathematics.
The existence of optimizers for \eqref{fieldinequality} is non-trivial and we will address this in another paper. Formally computing the Euler-Lagrange equations yields
\begin{equation}\label{eq:elpsi}
-i \sigma \cdot \nabla\left(  \frac{-i \sigma \cdot \nabla \psi}{|-i\sigma \cdot \nabla \psi|^{1/2}}\right) = e_s |\psi|\psi
\end{equation}
and 
\begin{equation}
\label{eq:ela}
{\rm curl}\left( \frac{{\rm curl} A }{|{\rm curl} A|^{1/2}}\right) = e_f |A|A \ ,
\end{equation}
where $e_s$ and $e_f$ are positive numbers. It is straightforward to check that the expressions \eqref{lossyau1} and \eqref{lossyau2} are solutions of these
equations.

On account of the non-linear nature of equations \eqref{eq:elpsi} and \eqref{eq:ela} one could choose $e_s=e_f=1$ but we choose not to do so. However, to make these expressions more palatable we set 
$$
\phi :=\frac{1} {\sqrt {e_s}} \frac{-i \sigma \cdot \nabla \psi}{|-i\sigma \cdot \nabla \psi|^{1/2}}
$$
so that \eqref{eq:elpsi} can be written as a system
$$
-i \sigma \cdot \nabla \phi = \sqrt{e_s} |\psi|\psi \ , \ -i \sigma \cdot \nabla \psi = \sqrt{e_s} |\phi|\phi \ .
$$
The same can be done with \eqref{eq:ela} by setting
$$
C = \frac{1}{\sqrt {e_f}} \frac{{\rm curl} A }{|{\rm curl} A|^{1/2}}
$$
and see that
$$
{\rm curl} C = \sqrt {e_f} |A|A \ , \ {\rm curl A} = \sqrt{e_f} |C| C \ .
$$
Thus, we end up with two pairs of dual equations. As we mentioned, we cannot say much about these systems of equations. Self-dual solutions to these equations are special solutions where $\phi= \psi$ and $C=A$ and once more, one can easily check that \eqref{lossyau1} and \eqref{lossyau2} satisfy 
\begin{equation} \label{selfdualspinor}
-i \sigma \cdot \nabla \psi = 3 |\psi|\psi
\end{equation}
and
\begin{equation} \label{selfdualfield}
{\rm curl} A = \frac43 |A|A \ .
\end{equation}
\begin{remark}
Since the spinor given by \eqref{lossyau1}   also satisfy the non-selfdual equations \eqref{eq:elpsi} (with $e_s = 3^{3/2}$)  and the field given by  \eqref{lossyau2} satisfies \eqref{eq:ela} (with $e_f = (\frac43)^{3/2}$) we venture the conjecture that they are optimizers for the inequalities given by \eqref{spinorinequality} and \eqref{fieldinequality}. If one accepts this conjecture one obtains $C_s = \frac32 |\Sp^3|^{1/3}=(3\,S_3)^{1/2}$ and $C_f = 2 |\Sp^3|^{1/3}=((16/3)S_3)^{1/2}$ and hence $\Vert B \Vert_{3/2} \ge 3 |\Sp^3|^{2/3} = 4 S_3$ as a necessary condition
for the existence on a zeromode. Thus, the truth of this conjecture would imply an improvement of the bound in Theorem \ref{magnetic} by a factor of $2$.
\end{remark}

Inequality \eqref{spinorinequality} is equivalent to the `integral' inequality
\begin{equation} \label{HLSSpin0}
\left| \left(\phi_1, \frac{1}{-i \sigma\cdot \nabla} \phi_2\right) \right| \le C_s^{-1}  \Vert \phi_1 \Vert_{3/2} \Vert \phi_2 \Vert_{3/2} \,.
\end{equation}
We now consider the simpler problem of finding the sharp constant in this inequality in the special case $\phi_1=\phi_2$, that is, finding the sharp constant $C$ in the inequality
\begin{equation} \label{HLSSpin}
\left| \left(\phi, \frac{1}{-i \sigma\cdot \nabla} \phi\right) \right|  \le C  \Vert \phi \Vert_{3/2}^2 \ .
\end{equation}
The operator$ \frac{1}{-i \sigma\cdot \nabla}$ has an integral kernel given by
$$
\frac{i}{4 \pi} \frac{\sigma \cdot x}{|x|^3}
$$
and hence the validity of \eqref{HLSSpin} for some constant follows from the Hardy-Littlewood-Sobolev inequality. We note that since this kernel is not positive definite, it is not clear that the optimal constant in \eqref{HLSSpin0} is achieved for $\phi_1=\phi_2$.

In fact, we shall consider this problem in any odd dimension $d \ge 3$.

\begin{theorem} \label{spinHLS}
Let $d\geq 3$ be odd and assume that there exists an optimizer $\phi \in L^{\frac{2d}{d+1}}(\R^d: \C^{2^\frac{d-1}{2}})$ for the inequality 
$$
\left| (\psi, [-i\gamma \cdot \nabla]^{-1} \psi) \right| \le  C \Vert \psi \Vert_{\frac{2d}{d+1}}^2 \ .
$$
Then $C=((d-2)/d)^{1/2} S_d^{-1/2}$ is the best possible constant and there is equality if
$$
\phi = \frac{1+i\gamma \cdot x}{(1+|x|^2)^{\frac{d+2}{2}}}|0\rangle  \ ,
$$
where $|0\rangle$ is a  well-chosen constant spinor.
\end{theorem}
\begin{proof}
A simple variation calculation shows that an optimizing spinor $\phi$, suitably normalized, satisfies the equation
$$
 \frac{1}{-i \gamma \cdot \nabla}\phi = |\phi|^{-\frac{2}{d+1}} \phi \ .
$$
Multiplying this equation by $\phi$ and integrating yields
$$
\left(\phi, \frac{1}{-i \gamma\cdot \nabla} \phi\right) = \Vert\phi \Vert_{\frac{2d}{d+1}}^{\frac{2d}{d+1}} \,,
$$
so
$$
C = \Vert \phi \Vert_{\frac{2d}{d+1}}^{-\frac{2}{d+1}} \ .
$$
If we set $\psi = d^{-\frac{d-1}{2}} |\phi|^{-\frac{2}{d+1}} \phi$ we find that $\psi \in L^{\frac{2d}{d-1}}(\R^d:\C^{2^{\frac{d-1}{2}}})$ is a solution of 
$$
-i\gamma \cdot \nabla \psi = d|\psi|^{\frac{2}{d-1}}\psi \ .
$$
Theorem \ref{spinorgeneral} then says that
$$
\frac{S_d}{d(d-2)} \le \Vert \psi \Vert_{\frac{2d}{d-1}}^{\frac{4}{d-1}} = \frac{1}{d^2}  \Vert \phi \Vert^{\frac{4}{d+1}}_{\frac{2d}{d+1}}
$$
and hence 
$$
C = \Vert \phi \Vert^{-\frac{2}{d+1}}_{\frac{2d}{d+1}}  \le \left( \frac{d-2}{d} \right)^{1/2} S_d^{-1/2} \ .
$$
It is easy to see that the spinor $\phi$ given in the theorem, with $|0\rangle$ chosen as the vacuum defined in Appendix \ref{dirac}, yields the inequality with the constant $((d-2)/d)^{1/2} S_d^{-1/2}$ and hence this constant is sharp.
\end{proof}


\section{Hijazi's approach} \label{hijaziapproach}

Theorem \ref{spinHLS} can also be viewed as a corollary of work by Hijazi \cite{Hijazi,Hijazi91}. The argument as it is presented {\it does not require the existence of an optimizer}, however, it requires that the spinors are $C^\infty$ and do not vanish, conditions that are not needed in our previous approach, which is based on the chain rule in Sobolev spaces. We present it because it is a very different and  interesting perspective and, while we think that one can remove the aforementioned conditions, this would only obfuscate the beauty of the reasoning. The approach rests on the conformal invariance of the functional \eqref{HLSSpin}, and it is therefore natural to explore this structure for the proof of Theorem \ref{spinHLS}. 

First we give some background. Consider the Dirac operator $D$ on a compact $d$ dimensional Riemannian manifold  with metric $g$ that carries a spin structure. The Lichnerowicz formula
$$
D^2 = \nabla^* \nabla + \frac14 R \,,
$$
where $R$ is the scalar curvature, leads to
\begin{equation} \label{form}
\int_M |D\psi|^2 \,d {\rm vol} = \int_M |\nabla \psi|^2 \,d {\rm vol} + \frac14 \int_M R|\psi|^2 \,d {\rm vol} \ .
\end{equation}
One sees from this formula, e.g., that if the scalar curvature is positive, then there is no harmonic spinor. In particular, there is a gap. To get a good lower bound on $\lambda_1(D)$, the eigenvalue of the Dirac operator that has smallest magnitude, one uses the projection
$$
T_X\psi = \nabla_X\psi +\frac1d X\cdot D\psi \ ,
$$
where the dot denotes Clifford multiplication. One finds that
\begin{equation} \label{square}
|\nabla \psi|^2 =|T\psi|^2 +\frac1d |D\psi|^2 \ .
\end{equation}
Note that this is essentially the same step as what we have used before with the introduction of the
projection $\Pi$. Using  \eqref{square} in \eqref{form} one gets
$$
\left(1-\frac1d \right) \int_M |D\psi|^2 \,d {\rm vol} = \int_M |T\psi|^2 \,d {\rm vol} + \frac14 \int_M R|\psi|^2  \,d {\rm vol}
$$
which yields the estimate
$$
\lambda_1(D)^2 \ge \frac{d}{4(d-1)} \inf_M R \ ,
$$
due to Friedrich \cite[Section 5.1]{Friedrich}.
In a further step, it was shown  in \cite{Hijazi} that for $d\ge 3$ 
$$
\lambda_1(D)^2 \ge  \frac{d}{4(d-1)}\,  \lambda_1(L) \ ,
$$
where $\lambda_1(L)$ is the lowest eigenvalue of the conformal Laplacian $L$, that is,
$$
\lambda_1(L) = \inf \frac{\int_M  \left( 4\frac{d-1}{d-2} |\nabla f|^2   +R f^2\right) d {\rm vol}}{\int_M f^2 d {\rm vol}} \ .
$$

If one changes the metric $g$ to the metric $g_u = e^{2u} g$ and denotes the Dirac operator in this new metric by $D_u$, then 
$$
D_u \psi_u = \left( e^{-\frac{d+1}{2}u}D e^{\frac{d-1}{2}u} \psi\right)_u \ .
$$
The spin bundles for $g$ and $g_u$ are isomorphic and $\psi_u$ is the image of $\psi$ under this isomorphism. Moreover, the conformal Laplacian changes to
$$
L_u = e^{-\frac{d+2}{2} u} L e^{\frac{d-2}{2}u} \,.
$$
Hence, in this context Hijazi's inequality reads
$$
\lambda_1(D_u)^2 \ge  \frac{d}{4(d-1)}  \lambda_1(L_u) \ ,
$$
where
$$
\lambda_1(L_u) = \inf  \frac{\int_M  \left( 4 \frac{d-1}{d-2} |\nabla (e^{\frac{(d-2)}{2} u} f)|^2 + R (e^{\frac{(d-2)}{2} u} f)^2\right) d {\rm vol}}{ \int_M f^2 e^{du} d{\rm vol}} \ .
$$
By H\"older's inequality,
$$
\int_M f^2 e^{du} d{\rm vol} \le \left( \int_M f^{\frac{2d}{d-2}} e^{du} d{\rm vol}\right)^{\frac{d-2}{d}} \left( \int_M e^{du} d{\rm vol}\right)^{\frac{2}{d}} \ ,
$$
which leads to the lower bound
\begin{align*}
	\lambda_1(D_u)^2 \left( \int_M  e^{du} d{\rm vol}\right)^{\frac{2}{d}} & \ge  \frac{d}{4(d-1)}  \inf_f  \frac{\int_M  \left( 4 \frac{d-1}{d-2} |\nabla (e^{\frac{(d-2)}{2} u} f)|^2 + R (e^{\frac{(d-2)}{2} u} f)^2\right) d {\rm vol}}{ \left( \int_M \left(e^{\frac{(d-2)}{2} u} f\right) ^{\frac{2d}{d-2}} d{\rm vol}\right)^{\frac{d-2}{d}} } \\
	& = \frac{d}{4(d-1)}  \inf_h  \frac{\int_M  \left( 4 \frac{d-1}{d-2} |\nabla h|^2 + R h^2 \right) d {\rm vol}}{ \left( \int_M h ^{\frac{2d}{d-2}} d{\rm vol}\right)^{\frac{d-2}{d}} } \ .
\end{align*}
The right side is a constant times the Yamabe constant of $(M,[g])$. This result was found in \cite{Hijazi91}.

If one applies the previous inequality to $M=\Sp^d$ with its standard metric and uses the sharp Sobolev inequality, one obtains the lower bound
 \begin{equation} \label{hijazibound}
 \lambda_1(D_u)^2   \left( \int_{\Sp^d}  e^{du} \, d{\rm vol}\right)^{\frac{2}{d}} \ge \frac{d}{d-2} \inf_h \frac{\int_{\Sp^d} \left(|\nabla h|^2 + \frac{d(d-2)}{4} h^2 \right) d {\rm vol} }
 {\left(  \int_{\Sp^3} h^{\frac{2d}{d-2}} \, d {\rm vol} \right)^{\frac{d-2}{d}}} = \frac{d^2}{4}\, |\Sp^d|^{2/d} \,.
\end{equation} 
In other words, we have that
$$
\Big | \int_{\Sp^d} \!\langle \phi, D_u^{-1} \phi\rangle e^{du} \,d {\rm vol} \Big | \le \frac{1}{|\lambda_1(D_u)|} \int_{\Sp^d} \!|\phi|^2 e^{du} \,d {\rm vol} \le\frac2d\,|\Sp^d|^{-1/d}\left( \int_{\Sp^d} \!e^{du} \,d{\rm vol} \right)^{1/d}\! \int_{\Sp^d}\! |\phi|^2 e^{du} \,d {\rm vol}.
$$
Since $D_u = e^{-\frac{d+1}{2}u} D e^{\frac{d-1}{2}u}$, we have that
$$
\Big | \int_{\Sp^d} \langle e^{\frac{d+1}{2} u} \phi, D^{-1}e^{\frac{d+1}{2}u} \phi\rangle  \,d {\rm vol} \Big | \le\frac2d\, |\Sp^d|^{-1/d}\left( \int_{\Sp^d} e^{du} \,d{\rm vol} \right)^{1/d} \int_{\Sp^d} |\phi|^2 e^{du} \,d {\rm vol}
$$
and, setting $e^{\frac{d+1}{2}u}\phi = \tilde \phi$, we find
$$
\Big | \int_{\Sp^d} \langle \tilde \phi, D^{-1} \tilde \phi\rangle  \, d {\rm vol} \Big | \le \frac2d\, |\Sp^d|^{-1/d}\left( \int_{\Sp^d} e^{du} \, d{\rm vol} \right)^{1/d}  \int_{\Sp^d} |\tilde \phi|^2 e^{-u} \,d {\rm vol} \,.
$$
If we choose $u$ such that $e^u = |\tilde \phi|^{\frac{2}{d+1}}$, then
$$
\Big | \int_{\Sp^d} \langle \tilde \phi, D^{-1} \tilde \phi\rangle  \,d {\rm vol} \Big | \le \frac2d\, |\Sp^d|^{-1/d} \left( \int_{\Sp^d} |\tilde \phi|^{\frac{2d}{d+1}} \,  d {\rm vol}\right)^{\frac{d+1}{d}}.
$$
Since $\Sp^d$ is conformally equivalent to $\R^d$, we obtain the desired inequality.
A similar statement ought to be true for the equation \eqref{selfdualfield}, but we were not able to 
adapt the methods to this case. This is an open problem.

Finally, let us remark that Hijazi's approach also sheds some light onto the proofs of Theorem \ref{spinor} and Theorem \ref{spinorgeneral}.
Assume we have a solution
$$
-i \gamma \cdot \nabla \psi  =  d \lambda \psi
$$
where we {\it assume} that $\lambda$ is nonnegative and sufficiently regular.
Define
$$
u = \ln \lambda
\qquad\text{and}\qquad
\phi = e^{- u (d-1)/2} \psi \,.
$$
Then, by the conformal transformation property $D_u = e^{-\frac{d+1}{2}u} D e^{\frac{d-1}{2}u}$, we have
$$
D_u \phi_u = d\, \phi_u,
$$
that is, the transformed operator $D_u$ has eigenvalue $d$. If one now applies inequality \eqref{hijazibound}, one obtains precisely the bound in Theorem \ref{spinorgeneral}. Thus, the proof of Theorem  \ref{spinorgeneral} given in  Section~\ref{sec:higherdim} has the advantage that it does not assume $\lambda$ to be nonnegative and, moreover, no regularity is assumed.


\appendix

\section{Some computations involving the Dirac matrices} \label{dirac}

The construction of zero modes in higher dimensions is more complicated and, as mentioned before, was accomplished by Dunne and Min \cite{Dunne-Min} using information about the Dirac equation on the sphere. The advantage of their construction is that it delivers automatically the dimension of the zero mode space. If one is satisfied with less information, then there is, we believe, a simpler way to construct the Dunne-Min zero modes. Moreover, it gives the opportunity to get acquainted with some of the properties of the Dirac matrices.  The basic idea is due to Adolf Hurwitz in his posthumously published paper `\"Uber die Komposition der quadratischen Formen' \cite{hurwitz}. In this paper he gave a complete classification of matrices $\gamma_j, j=1, \dots d$, satisfying the relations
$$
\gamma_j \gamma_k + \gamma_k \gamma_j = 2 \delta_{ij} \ .
$$
For our purposed we shall assume the the matrices $\gamma_j$ are self-adjoint in the space $\C^N$ with the usual inner product.

\begin{theorem} Let $d=2\nu+1$ or $d=2\nu$, and consider the  $N \times N$ hermitean matrices $\gamma_j, j=1, \dots, d$, satisfying
\begin{equation} \label{commrel}
\gamma_i \gamma_j + \gamma_j \gamma_i = 2 \delta_{ij} \ .
\end{equation} 
Then $N = 2^\nu$ and, if $\gamma'_j$ is another set of $2^\nu \times 2^\nu$ Hermitean matrices satisfying the same relations, then there exists a $2^\nu \times 2^\nu$ unitary matrix $A$ such that $\gamma'_j = A^*\gamma_j A$ for $j=1,\ldots,d$.
\end{theorem}

The proof proceeds by reducing the $\gamma$ matrices to a unitarily equivalent, but canonical set of matrices using an inductive procedure. 

\begin{corollary}\label{hurwitzcor}
Let $R$ be an $d\times d$ orthogonal matrix and define 
$$
\gamma'_j = \sum_{k=1}^d R_{jk} \gamma_k \ .
$$
Then there exists a unitary matrix $A$ such that for all $j=1,\ldots,d$ one has $\gamma'_j = A^* \gamma_j A$.
\end{corollary}
The computation with $\gamma$ matrices can be sometimes tedious and the following framework called `second quantization' is quite helpful. 

In the remainder of this section, \emph{we assume that $d=2\nu+1$ is odd}. 

We single out the matrix $\gamma_1$ and define
the `annihilation' and  `creation'  operators
$$
c_j := \frac12 (\gamma_{2j} + i\gamma_{2j+1}) \ ,\qquad  c^*_j = \frac12(\gamma_{2j} - i\gamma_{2j+1}) \ ,\qquad  j =1,2, \dots \nu \,,
$$
so that
$$
\gamma_{2j} = c_j+c_j^* \ ,\qquad \gamma_{2j+1} = \frac1i ( c_j - c^*_j) \ .
$$
One easily checks that
$$
c_jc^*_j +c^*_jc_j = I \ ,\qquad c_j^2 = c^{*2}_j = 0
$$
and, for $k \not= \ell$,
$$
c_kc_\ell +c_\ell c_k = 0 \ ,\qquad c^*_kc^*_\ell +c^*_\ell c^*_k = 0 \ ,\qquad  c_k c^*_\ell + c^*_\ell c_k =0 \ .
$$
Note that the matrix $\gamma_1$ is not involved in these definitions.

\begin{lemma}
There exists a vector $\phi\in \C^{2^\nu}$, a vaccum,  such that $\|\phi\|=1$ and
$$
c_j\phi= 0 \ ,\qquad j=1, \dots , \nu \ .
$$
\end{lemma}

\begin{proof}
Since $c_1^2=0$, it is clear that there exists $\phi\neq 0$ such that $c_1\phi=0$. Let $k$ be the first index such that
$c_k \phi \not= 0$. Setting $\psi = c_k \phi$ we see because of the commutation relations that $c_j \psi = 0,\ j=1, \dots, k-1$, and $c_k \psi = c_k^2 \phi = 0$. Thus, replacing $\phi$ by $\psi$ we have $c_j\psi = 0, j=1, \dots, k$. Continuing in this fashion we have a vector $\phi$ such that $c_k\phi=0$ for all $k=1, \dots, \nu$.
\end{proof}

\begin{lemma}\label{complete}
	Let $\vec \beta=(\beta_1, \dots, \beta_\nu) $ be a sequence with $\beta_j \in \{0,1\},\, j=1, \dots, \nu$. Then the vector
$$
|\vec \beta \rangle = c^{*\beta_1}_{k_1} \cdots c^{*\beta_\nu}_{k_\nu} \phi
$$
is non-zero if and only if the indices  $k_1, \dots, k_\nu$ are all distinct. In this case the vector is normalized. Moreover,
the vectors $|\vec \beta \rangle$ form an orthonormal basis in $\C^{2^\nu}$.
\end{lemma}

In view of this lemma, we will sometimes denote $\phi=|0\rangle$.

\begin{proof}
If one or more of the indices are not distinct, then by commuting the various operators results in a square of one of the $c_i^*$, which is zero. Hence we may assume that the indices  $k_1, \dots, k_\nu$ are all distinct. We also may assume that $\beta_1 = 1$ because otherwise $c_{k_1}^{\beta_1} = I$ and we may move on to the next index. We have
$$
\Vert c^{*\beta_1}_{k_1} \cdots c^{*\beta_\nu}_{k_\nu} \phi\Vert^2 = \left( | 0\rangle , c^{\beta_\nu}_{k_\nu} \cdots c^{\beta_1}_{k_1}c^{*\beta_1}_{k_1} \cdots c^{*\beta_\nu}_{k_\nu} \phi\right)
$$
and using $c^{\beta_1}_{k_1}c^{*\beta_1}_{k_1} = I - c^{*\beta_1}_{k_1}c^{\beta_1}_{k_1} $ we find
$$
\left(| 0\rangle, c^{\beta_\nu}_{k_\nu} \cdots c^{\beta_1}_{k_1}c^{*\beta_1}_{k_1} \cdots c^{*\beta_\nu}_{k_\nu} \phi\right) = \left(| 0\rangle, c^{\beta_\nu}_{k_\nu} \cdots c^{\beta_2}_{k_2}c^{*\beta_2}_{k_2} \cdots c^{*\beta_\nu}_{k_\nu} \phi\right)
 -\left(| 0\rangle, c^{\beta_\nu}_{k_\nu} \cdots c^{*\beta_1}_{k_1}c^{\beta_1}_{k_1} \cdots c^{*\beta_\nu}_{k_\nu} \phi\right) \ .
$$
The second term on the right side vanishes because the indices are distinct  and thus $c^{\beta_1}_{k_1}$ either commutes or anti-commutes with all the matrices on the right and once it hits $\phi$ it yields zero. In this fashion we may move the
annihilation matrices to the right and obtain that this state is normalized.
Incidentally this also makes it clear that the state vanishes if two indices are the same on account of the fact that
$c_j^2 = c^{*2}_j=0$. From this argument it also follows that for $\vec \beta \not = \vec \beta'$ 
$$
\left( |\vec \beta\rangle, |\vec \beta'\rangle\right) = 0
$$
and hence we have $2^\nu$ orthonormal vectors which constitute an orthonormal basis.
\end{proof}

\begin{lemma}\label{unique}
The vacuum is unique (up to a constant phase).
\end{lemma}

\begin{proof}
Suppose that $v$ is another vacuum, i.e.~$\|v\|=1$ and for all $\alpha=1, \dots, \nu$,
$$
c_\alpha v = 0 \ .
$$
We may assume that $\langle 0|v\rangle = 0$.
Then 
$$
\left(v,  c^{*\beta_1}_{k_1} \cdots c^{*\beta_\nu}_{k_\nu} \phi\right)
$$
is always zero and therefore, by Lemma \ref{complete}, $v=0$, which is a contradiction.
\end{proof}

We note that $\gamma_1 \phi$ satisfies the same properties as $\phi$, namely, $\|\gamma_1\phi\|=1$ and
$$
c_\alpha \gamma_1 \phi= -  \gamma_1 c_\alpha\phi= 0
$$
for all $\alpha$. By the uniqueness result of Lemma \ref{unique} there is a $\theta\in\R$ such that $\gamma_1 \phi= e^{i\theta} \phi$. Since $\gamma_1$ is self-adjoint, we have $e^{i\theta}=\pm 1$. In case it is $-1$, we can change the sign of $\gamma_1$ without changing the commutation relations and arrive at the same relation with $+1$. Hence we may adopt the convention that $\gamma_1\phi= \phi$.

The point about introducing this formalism is the following result.
\begin{lemma}
Introduce a $(2\nu+1)\times(2\nu+1)$ matrix $\omega$ with entries
$$
\omega_{\alpha,\beta} =
\begin{cases}
0 & \text{if}\ \alpha=1 \ \text{or}\ \beta =1 \ \text{or}\ \alpha=\beta \,, \\
\langle 0 | i \gamma_{\alpha }\gamma_{\beta} |0\rangle & \text{otherwise} \,.
\end{cases}
$$
Then
$$
\omega = {\rm diag} ( 0, -i\sigma_2,\ldots, -i\sigma_2) \,,
$$
where the zero is a number and there are $\nu$ $2\times 2$-blocks $i\sigma_2$.
\end{lemma}
\begin{proof}
Since $\omega$ is skew and vanishes on the diagonal, it suffices to compute $\omega_{\alpha,\beta}$ when $\alpha<\beta$. Moreover, since $\gamma_1 |0\rangle =0$ and $\gamma_1$ is selfadjoint, we have $\omega_{1,\beta}=0$ for all $\beta>1$. For the remaining entries, we need to distinguish whether $\alpha$ and $\beta$ are even or odd. When both are even, we have for $1\leq j< k\leq\nu$,
$$
\omega_{2j,2k} =  \langle 0| i(c_j+c_j^*)(c_k+c_k^*)|0\rangle = i \langle 0|c_jc^*_k |0\rangle = 0 \ .
$$
Similarly, when both are odd, we have for $1\leq j<k\leq\nu$,
$$
\omega_{2j+1,2k+1} =  - \langle 0| i(c_j-c_j^*)(c_k-c_k^*)|0\rangle = - i \langle 0|c_jc^*_k |0\rangle = 0 \ .
$$
Next, we consider $\alpha$ is even and $\beta$ is odd. If $\alpha=\beta-1$, we get
$$
\omega_{2j,2j+1} = \langle 0 |  (c_j+c_j^*)(c_j -c^*_j) |0\rangle  = -1 \ .
$$
Otherwise, for $1\leq j<k\leq\nu$,
$$
\omega_{2j,2k+1} = \langle 0 |  (c_j+c_j^*) (c_k-c^*_k) |0\rangle = 0 \,.
$$
Finally, we have the case where $\alpha$ is odd and $\beta$ is even. For $1\leq j<k\leq\nu$, we get
$$
\omega_{2j+1,2k} = \langle 0 |  (c_j-c_j^*) (c_k+c^*_k) |0\rangle = 0 \,.
$$
This proves the claimed formula for the entries of the matrix $\omega$.
\end{proof}

After these preliminaries we discuss now an alternative approach to the Dunne--Min generalization \cite{Dunne-Min} of \cite{LossYau}. The following example is relevant. It is the higher dimensional analog of choice for the vector potential in \cite{LossYau}. 
Consider
\begin{equation} \label{spinvector}
((1+i \vec \gamma \cdot x)\eta,  \vec \gamma (1+i \vec \gamma \cdot x)\eta) \ ,
\end{equation}
where $\eta\in\C^{2^\nu}$ is normalized. Recall that for $d=3$ the $\gamma$ matrices are the Pauli matrices and there is the well know identity
$$
|(\eta, \vec \sigma \eta)|^2 = |\eta|^4 \ .
$$
This leads to the identity
$$
\vec \sigma \cdot (\eta, \vec \sigma \eta) \eta = \eta \ ,
$$
which is very useful  for constructing zero modes. It turns out that this identity also holds for $d=5$, but not in higher dimensions. In particular, it does not hold for 
\eqref{spinvector} for general $\eta$.
 
Things simplify considerably if we choose the constant spinor $\eta$ to be the vaccum $\phi$. We compute
\begin{align*}
	(1 - i \vec \gamma \cdot x) \gamma_j  (1+i \vec \gamma \cdot x)  
	& = \gamma_j -i \vec \gamma \cdot x \gamma_j + i \gamma_j \vec \gamma \cdot x + \vec \gamma \cdot x \gamma_j \vec \gamma \cdot x \\
	& = \gamma_j - 2i x \cdot \vec \gamma  \gamma_j + i x \cdot (\gamma_j \vec \gamma + \vec \gamma \gamma_j) - (x \cdot \vec \gamma )^2 \gamma_j + x \cdot \vec \gamma x \cdot (\gamma_j \vec \gamma + \vec \gamma \gamma_j) \\
	& = \gamma_j - 2i x \cdot \vec \gamma  \gamma_j + 2i x_j  -|x|^2\gamma_j + 2x \cdot \vec \gamma x_j \\
	& =(1-|x|^2) \gamma_j  + 2x \cdot \vec \gamma x_j  - 2i x \cdot \vec \gamma  \gamma_j + 2i x_j  \,.
\end{align*}
Taking expectation we get
$$
((1+i \vec \gamma \cdot x)\phi,  \gamma_j  (1+i \vec \gamma \cdot x)\phi) = 
(1-|x|^2) (\phi, \gamma_j \phi) + 2 x \cdot (\phi, \vec \gamma \phi) x_j  - 2 \sum_{k \not= j, 1} x_k (\phi, i\gamma_k \gamma_j \phi) \ .
$$
Since $\gamma_1 \phi = \phi$, we find that
$(\phi, \gamma_j \phi) = 0, j\not=1$. Hence we have that for this particular state
$$
((1+i \vec \gamma \cdot x)\phi,  \gamma_1 (1+i \vec \gamma \cdot x)\phi) = (1-|x|^2 +2x_1^2)|\phi|^2 \ .
$$
For the component $j\not= 1$ we find
$$
((1+i \vec \gamma \cdot x)\phi,  \gamma_j  (1+i \vec \gamma \cdot x)\phi) =  (2 x_1x_j + 2 [\omega x]_j ) |\phi|^2 \,.
$$
Here $\omega$ is the $(2\nu+1) \times (2\nu+1)$ skew matrix introduced above.
We introduce the field
$$
U_j(x) := \frac{1}{1+|x|^2}\left( (1 + i \vec \gamma \cdot x)\phi,  \gamma_j  (1+i \vec \gamma \cdot x)\phi\right)
=
\begin{cases}
	\frac{1}{1+|x|^2} (1-|x|^2 +2x_1^2) & \text{if}\ j=1 \,, \\
	\frac{1}{1+|x|^2} \left( 2 x_1x_j  + 2 [\omega x]_j\right) & \text{if}\ j \neq 1 \,.
\end{cases}
$$
This can be written more concisely as
$$
U(\vec x) =  \frac{1}{1+|x|^2} \left( (1-|\vec x|^2) \vec e_1 +2 (\vec e_1 \cdot \vec x) \vec x+ 2 \omega \vec x\right),
$$
where $[\omega \vec x]_k = \sum_{j=1}^d \omega_{kj} x_j $.
A straightforward computation shows that
\begin{align*}
	|U(\vec x)|^2 & =  \frac{1}{(1+|x|^2)^2} \left( (1-|\vec x|^2) \vec e_1 +2 (\vec e_1 \cdot \vec x) \vec x+ 2 \omega \vec x\right)^2 \\
	& = \frac{1}{(1+|x|^2)^2} \left( (1-|\vec x|^2)^2 + 4 (\vec e_1 \cdot \vec x)^2 |x|^2 + 4 |\omega \vec x|^2 + 4(1-|x|^2) (\vec e_1 \cdot \vec x)^2 \right).
\end{align*}
Since $ |\omega \vec x|^2 = (\omega \vec x, \omega \vec x) = (\vec x, \omega^T \omega \vec x) = \sum_{j=2}^d x_j^2$
we get
$$
|U(\vec x)|^2 = 1 \,.
$$
In other words, the vector
$$
\vec U(\vec x) =\left( \frac{1+i\vec x \cdot \vec \gamma}{(1+|x|^2)^{1/2}} \phi, \vec \gamma  \frac{1+i\vec x \cdot \vec \gamma)}{(1+|x|^2)^{1/2}} \phi\right)
$$
is a unit vector. Now consider the self adjoint matrix
$$
M:= \vec U(\vec x) \cdot \vec \gamma \ ,
$$
whose square is $|\vec U(\vec x)|^2 = 1$. Hence the eigenvalues of $M$ are $\pm 1$. Moreover
$$
\left( \frac{1+i\vec x \cdot \vec \gamma}{(1+|x|^2)^{1/2}} \phi, M \frac{1+i\vec x \cdot \vec \gamma}{(1+|x|^2)^{1/2}} \phi\right) = |\vec U(\vec x)|^2= 1
$$
and hence  we have that
$$
M \frac{1+i\vec x \cdot \vec \gamma}{(1+|x|^2)^{1/2}} \phi = \frac{1+i\vec x \cdot \vec \gamma}{(1+|x|^2)^{1/2}} \phi \ .
$$
If we set
$$
\psi :=  \frac{1+i\vec x \cdot \vec \gamma}{(1+|x|^2)^{d/2}} \phi  \ ,
$$
then a simple computation yields
$$
-i \vec \gamma \cdot \nabla \psi =  \frac{d}{1+|x|^2} \psi
$$
and if we define 
$$
A(x) :=  \frac{d}{1+|x|^2} U(x) \ ,
$$
then
$$
-i \vec \gamma \cdot \nabla \psi = \vec \gamma \cdot A \psi
$$
and we have constructed our zero modes.


\section{Generalization of the spin-field interaction term  to arbitrary dimensions} \label{genBfield}

Squaring the Dirac equation yields
\begin{align*}
	[ \gamma \cdot (-i\nabla -A) ]^2 & = \sum_{jk} \gamma_j \gamma_k (-i\partial_j -A_j)(-i\partial_k -A_k) \\
	& =\sum_j \gamma_j \gamma_j (-i\partial_j -A_j)(-i\partial_j-A_j) + \sum_{j\not=k} \gamma_j \gamma_k (-i\partial_j -A_j)(-i\partial_k -A_k) \,.
\end{align*}
We have
$$
\sum_j \gamma_j \gamma_j (-i\partial_j -A_j)(-i\partial_j-A_j) = (-i \nabla -A)^2
$$
and
\begin{align*}
	& \sum_{j\not=k} \gamma_j \gamma_k (-i\partial_j -A_j)(-i\partial_k -A_k) \\
	& = \frac12  \sum_{j\not=k} \gamma_j\gamma_k (-i\partial_j -A_j)(-i\partial_k -A_k) + \frac 12 \sum_{j\not=k} \gamma_k \gamma_j (-i\partial_k -A_k)(-i\partial_j -A_j) \\
	& =  \frac12  \sum_{j\not=k} \gamma_j\gamma_k \left[ (-i\partial_j -A_j)(-i\partial_k -A_k) - (-i\partial_k -A_k)(-i\partial_j -A_j)\right] \\
	& = \frac{i}{2}  \sum_{j\not=k} \gamma_j\gamma_k \left[ \partial_jA_k - \partial_k A_j \right].
\end{align*}
For each fixed $x\in\R^d$, the matrix $B_{jk} :=  \partial_jA_k - \partial_k A_j$ is an antisymmetric matrix and  there is an orthogonal matrix $R$ (depending on $x$) such that
$$
R^T B R = D
$$
where 
$$
D=
\begin{cases}
	{\rm diag} (D_1i\sigma_2,\ldots,D_\nu i\sigma_2,0) & \text{if}\ d=2\nu+1 \ \text{is odd} \,,\\
	{\rm diag} (D_1i\sigma_2,\ldots,D_\nu i\sigma_2) & \text{if}\ d=2\nu \ \text{is even} \,.
\end{cases}
$$
Here there are $\nu$ $2\times 2$ blocks $i\sigma_2$ and, if $d$ is odd, an additional $1\times 1$ `block' consisting of the number $0$. For instance, in 5 dimensions
$$
\left[\begin{array}{ccccc} 0 & D_1 & 0 & 0 &  0 \\-D_1 & 0 & 0 & 0 & 0 \\  0 & 0 & 0 & D_2 & 0 \\  0 & 0 & -D_2& 0 & 0 \\ 0 & 0  & 0 & 0 & 0 \end{array}\right] \ .
$$

Since the trace of $B$ is zero, we have
$$
 \sum_{j} \gamma_j\gamma_j B_{jj} = 0 \,.
$$
 Hence
\begin{align*}
	\frac{i}{2}  \sum_{j\not=k} \gamma_j\gamma_k B_{jk}  & = \frac{i}{2}  \sum_{j k} \gamma_j\gamma_k B_{jk} =  \frac{i}{2} \sum_{\alpha \beta}  \sum_{jk} \gamma_j\gamma_k R_{j\alpha} D_{\alpha \beta} R_{k \beta} \\
	& = \frac{i}{2}  \sum_{\alpha \beta} (  \sum_{j} \gamma_j R_{j\alpha}) ( \sum_k  \gamma_k R_{k\beta}) D_{\alpha \beta} \ .
\end{align*} 
If we set
$$
\Gamma_\beta :=  \sum_{j} \gamma_j R_{j\beta} \ ,
$$
then we have 
\begin{align*}
	\Gamma_\alpha \Gamma_\beta + \Gamma_\beta \Gamma_\alpha & =\sum_{jk}  \gamma_j \gamma_k R_{j \alpha} R_{k \beta}
	+ \sum_{jk}  \gamma_k \gamma_j R_{k \beta} R_{j \alpha} =  2\sum_{jk}  \delta_{jk}  R_{k \beta} R_{j \alpha} \\
	& =2\sum_j  R_{j \beta} R_{j \alpha}  = 2(R^TR)_{\beta, \alpha} = 2 \delta_{\alpha \beta}.
\end{align*}
Hence, according to Corollary \ref{hurwitzcor},  there exists a unitary matrix $U$ such that
$$
\Gamma_\alpha = U^* \gamma_\alpha U \ ,
\qquad \alpha=1,\ldots, d \ ,
$$
and we can write
$$
 \frac{i}{2}  \sum_{j\neq k} \gamma_j\gamma_k B_{jk}  = U^*  \frac{i}{2}  \sum_{\alpha \beta} \gamma_\alpha\gamma_\beta D_{\alpha\beta}  U
 = U^* i  \left[\gamma_1 \gamma_2 D_1 + \gamma_3\gamma_4 D_2 +\dots
 + \gamma_{2\nu-1} \gamma_{2\nu} D_\nu \right] U \,.
 $$
 The matrices $\gamma_1\gamma_2$ and $\gamma_3\gamma_4$ etc, are skew symmetric,  commute with each other  and we can
 simultaneously diagonalize them by a unitary matrix $V$, that is, 
$$
\gamma_{2k-1} \gamma_{2k} = -i V^* \Sigma_{2k-1, 2k} V \ ,
\qquad k=1,\ldots,\nu \ \
$$
with diagonal matrices $\Sigma_{2k-1 2k}$. Since $(\gamma_i \gamma_j)^2 =-1$ the eigenvalues of $\Sigma_{2k-1,2k}$ must be $\pm 1$.
Thus, all things considered, we get
 \begin{equation*}
  \frac{i}{2}  \sum_{j\neq k} \gamma_j\gamma_k B_{jk} =(VU)^* \left[ \Sigma_{12} D_1+ \Sigma_{34} D_3 + \cdots \Sigma_{2\nu-1, 2\nu} D_{\nu} \right] (VU)
 \end{equation*}
 where the matrices $\Sigma_{i,i+1}$ are diagonal and have $\pm 1$ in the diagonal.  Thus, if $\psi$ is a spinor, we have that
 $$
\left |  \langle \psi,  \frac{i}{2}  \sum_{j\neq k} \gamma_j\gamma_k B_{jk} \psi \rangle  \right | \le |\psi|^2 \sum_{k=1}^{\nu} |D_k|
$$
This fits with the three dimensional case where $\nu=1$ and $|D_1|=|B|$.

Moreover, we have
$$
\sum_{k=1}^{\nu} |D_k| \leq \nu^{1/2} \left( \sum_{k=1}^\nu |D_k|^2 \right)^{1/2} = \nu^{1/2} \left( \sum_{j<k} |B_{jk}|^2 \right)^{1/2}.
$$
The last identity comes from the fact that conjugation by an orthogonal matrix $R$ does not change the Hilbert--Schmidt norm of the matrix $D=R^TBR$.



\end{document}